\documentclass[a4paper,11pt]{amsart}
\usepackage{amsmath,inputenc,euscript,amssymb,geometry}
\usepackage{graphicx}
\usepackage{amssymb}
\usepackage{latexsym}
\usepackage{amssymb,amsbsy,amsmath,amsfonts,amssymb,amscd,color}
\usepackage{tikz}
\usepackage{mathrsfs}
\usepackage{enumitem}

\usepackage{hyperref}

\DeclareFontFamily{U}{matha}{\hyphenchar\font45}
\DeclareFontShape{U}{matha}{m}{n}{
  <-6> matha5 <6-7> matha6 <7-8> matha7
  <8-9> matha8 <9-10> matha9
  <10-12> matha10 <12-> matha12
  }{}
\DeclareSymbolFont{matha}{U}{matha}{m}{n}
\DeclareMathSymbol{\Lt}{3}{matha}{"CE}

\newtheorem*{thm}{Theorem}
\newtheorem{lemma}{Lemma}[section]
\newtheorem{theorem}[lemma]{Theorem}
\newtheorem{prop}[lemma]{Proposition}
\newtheorem{corol}[lemma]{Corollary}
\newtheorem{rem}[lemma]{Remark}
\newtheorem{remark}[lemma]{Remark}
\newtheorem{example}[lemma]{Example}
\newtheorem{definition}[lemma]{Definition}
\newcommand*{\avint}{\mathop{\ooalign{$\int$\cr$-$}}}
\numberwithin{equation}{section}
\DeclareMathOperator{\curl}{curl}
\DeclareMathOperator{\divbis}{div}
\begin{document}
\title[MIcrolocal analysis of singular measures]{MIcrolocal analysis of singular measures}
\author[Valeria Banica]{Valeria Banica}
\address[Valeria Banica]{Sorbonne Universit\'e, CNRS, Universit\'e de Paris, Laboratoire Jacques-Louis Lions (LJLL), F-75005 Paris, France, and Institut Universitaire de France (IUF)}
\email{Valeria.Banica@sorbonne-universite.fr} 

\author[Nicolas Burq]{Nicolas Burq}
\address[Nicolas Burq]{Laboratoire de mathÃ©matiques d'Orsay, CNRS, Universit\'e Paris-Saclay, B\^at.~307, 91405 Orsay Cedex, France, and Institut Universitaire de France (IUF)}
\email{nicolas.burq@universite-paris-saclay.fr } 
\maketitle
 \begin{abstract}
The purpose of this article is to investigate the structure of singular  measures from a microlocal perspective. Motivated by the result of De Philippis-Rindler \cite{DePRi}, and the notions of wave cones of Murat-Tartar~\cite{Mu1, Mu2, Ta1, Ta2} and  of polarisation set of Denker \cite{De} we introduce a notion of $L^1$-regularity wave front set for scalar and vector distributions. Our main result is a proper microlocal characterisation of the support of the singular part of tempered Radon measures and of their polar functions at these points. The proof is based on De Philippis-Rindler's approach reinforced by microlocal analysis techniques and some extra geometric measure theory arguments. 
We deduce a sharp $L^1$ elliptic regularity result  which appears to be new even for scalar measures and which enlightens the interest of the techniques from geometric measure theory to the study of harmonic analysis questions. For instance we prove that $ \Psi^0  L^1\cap \mathcal M_{loc}\subseteq L^1_{loc},$ and in particular we obtain $L^1$ elliptic regularity results as $\Delta u\in L^1_{loc}, D^2 u \in \mathcal M_{loc} \Longrightarrow  D^2  u\in L^1_{loc}.$ We also deduce several consequences including extensions of the results in \cite{DePRi} giving constraints on the polar function at singular points for measures constrained by a PDE, and of Alberti's rank one theorem. Finally, we also illustrate the interest of this microlocal approach with a result of propagation of singularities for constrained measures. \\ MSC classification codes: primary 42B37 (Harmonic analysis and PDEs), secondary 28B05 (Vector-valued set functions, measures and integrals), 35A18 (Wave front sets in context of PDEs), 35Jxx (Elliptic equations and elliptic systems). 

\end{abstract}
 	
 \section{Introduction}


\subsection{The framework}
We start by recalling some classical facts about the class of locally bounded Radon measures on $\mathbb R^d$ with values in $\mathbb R^m$, $\mathcal{M}_{loc} ( \mathbb{R}^d, \mathbb{R}^m)$  (see for instance \cite{Ru} or \cite{M}). The
Radon-Nikodym Theorem allows to write the polar decomposition $d\mu=\frac{d\mu}{d|\mu|}d|\mu|$, where the non negative  measure $|\mu|$ is the total variation of the measure $\mu$ and the function $\frac{d\mu}{d|\mu|}\in L^1_{loc}(\mathbb{R}^d (d|\mu|), \mathbb S^{m-1})$ is the Radon-Nikodym derivative of $\mu$ with respect to $|\mu|$, called the polar function of $\mu$. Moreover, the Radon-Nikodym Theorem gives the Lebesgue decomposition of $|\mu|$ with respect to the Lebesgue measure, $\mathcal L^d$, so we have
$$d\mu=g d\mathcal L^d+\frac{d\mu}{d|\mu|}d|\mu|_s,$$
where $g\in L^1(\mathbb R^d,\mathbb R^m)$ and the positive measure $|\mu|_s$ satisfies $|\mu|_s\perp \mathcal L^d$. Finally, the Radon-Nikodym Theorem applied to each component of $\mu$ with respect to $\mathcal L^d$ implies that the singular part of $\mu$ is $d\mu_s=\frac{d\mu}{d|\mu|\,}d|\mu|_s$, and in particular $\frac{d\mu}{d|\mu|\,}=\frac{d\mu_s}{d|\mu|_s\,}$, $\mu_s$ a.e. and $|\mu_s|=|\mu|_s$. 

In the following discussion and results we shall consider pseudodifferential operators and we have gathered in Section~\ref{sec.pseudo} the definitions and basic results from the theory we need. We denote by $S^k_{\text{cl}}(\mathbb R^d)$ the class of symbols of order $k$ and by $\Psi^k_{\text{cl}}(\mathbb R^d)$ the class of classical  pseudodifferential operators of order $k$. For $u$  a tempered distribution on $\mathbb{R}^d$ with values in $\mathbb R^m$, that is $u=(u_1,...,u_m)$ with $u_i\in\mathcal S'(\mathbb{R}^d,\mathbb R)$, if $A=(A_{ij})$ is a $n\times m$ matrix of pseudodifferential operators then the distribution $Au\in\mathcal D'(\mathbb{R}^d,\mathbb R^n)$ is defined by
$$(Au)_j=\sum_k A_{jk} u_k , \quad 1\leq j\leq n.$$
For $A\in\Psi^k_{\text{cl}}(\mathbb R^d)$ we denote by $a_k:\mathbb{R}^d\times\mathbb R^{d}\rightarrow\mathbb R$ (or simply by $a$) its homogeneous principal symbol and similarly for matrices of pseudodifferential operators.\\

We recall now the following result of De Philippis and Rindler that connects the polarisation of the singular part of an $\mathcal A$-constrained measure and the wave cone of non-elliptic directions of the operator $\mathcal A$ introduced by Murat and Tartar~\cite{Mu1, Mu2, Ta1, Ta2}:
$$\Lambda_{\mathcal A}:=\underset{|\xi|=1}{\bigcup}\ker (\sum_{|\alpha|=k}A_\alpha\xi^\alpha).$$

\begin{thm} [{\cite[Theorem 1.1]{DePRi}}] Let $\mathcal A=\sum_{|\alpha|\leq k}A_\alpha\partial^\alpha,  k\in\mathbb N $, $A_\alpha\in M_{n\times m}$, and $\mu\in\mathcal{M}_{loc} ( \mathbb{R}^d, \mathbb{R}^m)$. Then\\
$$\mathcal A\mu\overset{\mathcal D'}{=}0_{\mathbb R^n}\quad \Longrightarrow \quad
\frac{d\mu}{d|\mu|}(x)\overset{|\mu|_s-a.e}{\in}\Lambda_{\mathcal A}.$$
\end{thm}

This result is impressive by its statement, its proof and its deep consequences in geometric measure theory proved in \cite{DePRi} by taking the constraint operator $\mathcal A$ to be $div, curl$ and $curl curl$. We shall do now two series of remarks, one with respect to elliptic constraints, and one on the link with the wave front set notions.

First we notice that  \cite[Theorem 1.1]{DePRi}, even for variable coefficients differential operators, is only relevant for systems ($m\geq 2$),   and non elliptic equations, i.e. the operator $\mathcal A$ is {not elliptic}. Indeed, for the scalar case, and non elliptic equation, \cite[Theorem 1.1]{DePRi} gives no information. On the other hand, 
 if the  differential operator $\mathcal{A}$ is elliptic and of order $0$, even space dependent, i.e it is the multiplication by an invertible space dependent matrix, it is straightforward that 
$$ \mathcal{A}(x) \mu = f\in L^1  \Rightarrow \mu= \mathcal{A}(x) ^{-1} f \in L^1,$$
so $\mu$ has no singular part. Finally, if $\mathcal{A}(x, D_x)$ is elliptic of order $k>0$, $\mathcal{A} \mu = f \in \mathcal{M}$, then the pseudodifferential calculus shows, after replacing $\mathcal A$ by $\mathcal A^*\mathcal A$ if $m\neq n$,  that there exists an operator $\mathcal{B}\in M_{m\times m} (\Psi^{-k}_{cl} ( \mathbb{R}^d))$, such that $\mathcal{B} \mathcal{A} = \text{Id} +\mathcal{R},$ where $\mathcal R$ maps $\mathcal{S}'(\mathbb{R}^n) $ to $C^\infty( \mathbb{R}^n)$. As a consequence, we get
$$ \mu = \mathcal{B} \mathcal{A} \mu- \mathcal{R}\mu =  \mathcal{B} f- \mathcal{R}\mu, \qquad f\in \mathcal{M}, $$ 
and since pseudodifferential operators of negative order send $\mathcal{M}$ to $L^1$ (see Corollary~\ref{coro})  we also get that $\mu$ has no singular part.

We now notice that the general scalar case, when $\Lambda_{\mathcal A}$ corresponds to the characteristic set, can be seen also by using H\"ormander's theorem:
 $$\mu\notin L^1 \mbox{ at } x\Rightarrow \mu\notin\mathcal C^\infty\mbox{ at } x\Rightarrow \exists \xi, (x,\xi)\in WF_{\mathcal C^\infty}\mu\overset{\mathcal A\mu=0}{\Rightarrow} a_k(x,\xi)=0.$$
We recall that the classical wave front set of a distribution $u\in\mathcal D(\mathbb R^d,\mathbb R)$, which can be written as: 
$$WF(u):=\{(x,\xi)\in \mathbb{R}^d\times\mathbb R^{d*};\, \mathbb R= \cap_{Au\in\mathcal C^\infty, A\in\Psi(\mathbb R^d)}\ker a(x,\xi)\},$$
introduced by H\"ormander gives a proper characterisation of the singular support of a distribution,  i.e. the complementary of points where the distribution is $\mathcal C^\infty$. 
This lead us to looking for a wave front set notion appropriate for the present context. 
For vectorial distributions $u\in\mathcal D(\mathbb R^d,\mathbb R^m)$, a refinement of the classical $\mathcal C^\infty$-regularity wave front set that takes into account the polarization was introduced by Denker \cite{De} following suggestions by  of H\"ormander:
$$WF_{pol}(u):=\{(x,\xi,\omega)\in \mathbb{R}^d\times\mathbb R^{d*}\times\mathbb R^{m};\, \omega\in \cap_{Au\in\mathcal C^\infty, A\in M_{1\times m}(\Psi(\mathbb R^d))}\ker a(x,\xi)\}.$$
In particular, Proposition 2.7 in \cite{De} on the action of the principal symbol on the fiber of the polarization wave front set encouraged us to look for a purely microlocal setting (see our  main result Theorem \ref{thWFL1}). Moreover, results of propagations of $C^\infty$ or $L^2$-based singularities are available for the polarisation wave fronts sets in ~\cite{De} (see also G\'erard~\cite{Ge85}).

Summarizing the last remarks, the starting motivation for us to introduce a $L^1$- wave front notion was to make a connection between De Phillipis-Rindler results and the polarisation wave fronts introduced in microlocal analysis by Dencker~\cite{De}. The proof of \cite[ Theorem 1.1]{DePRi}, involving powerful measure theory notions, will allow us to prove our main result Theorem \ref{thWFL1} and to treat critical $L^1$-cases that were not at reach by the classical H\"ormander theory.


\subsection{Definition of a $L^1$-wave front set for distributions and the main result}

In this article we introduce an $L^1$-regularity wave front set for vector valued  tempered distributions.
For simplicity, we chose to work with tempered distributions and tempered Radon measures rather than distributions and locally bounded Radon measures, which would have been possible by localising our definitions of wave fronts with a cut off (see Remark~\ref{front-onde-rem} 1)

 \begin{definition}\label{defWFL1} For $u\in\mathcal {S}'(\mathbb{R}^d,\mathbb R^m)$ we define the set

$$WF_{L^1}(u )\subseteq  (T^*\mathbb{R}^d \setminus\{0\}) \times (\mathbb{R}^m  \setminus\{0\}))= \mathbb{R}^d\times\mathbb R^{d*}\times\mathbb R^{m*}$$ by the following:  
$(x_0, \xi_0, \omega_0) \notin WF_{L^1}(u )$ if and only if there exist $A \in M_{1\times m}(\Psi^0_{\text{cl}}(\mathbb R^d))$, $N\in\mathbb N^*, Q_i\in \Psi^0_{\text{cl}}(\mathbb R^d)$ and $ f_i \in L^1( \mathbb{R}^d,\mathbb R)$ for $1\leq i\leq N$,  
satisfying 
\begin{equation}\label{definiWF} 
 A u = \sum_{i=1}^N Q_i f_i,
 \end{equation}
with $A$ elliptic at $(x_0, \xi_0, \omega_0)$ in the sense that 
$$ a_0(x_0, \xi_0)\,\omega_0\neq 0.$$
\item This definition is equivalent to: 
\begin{equation}\label{definiWFdir} WF_{L^1}(u )=\{(x,\xi,\omega)\in \mathbb{R}^d\times\mathbb R^{d*}\times\mathbb R^{m*};\, \omega\in \cap\ker a_0(x,\xi)\}, 
 \end{equation}
 where the intersection is over all possible $A\in M_{1\times m}(\Psi^0_{\text{cl}}(\mathbb R^d))$, $N\in\mathbb N^*, Q_i\in \Psi^0_{\text{cl}}(\mathbb R^d)$ and $f_i \in L^1( \mathbb{R}^d, \mathbb R)$ for $1\leq i\leq N$, such that 
 $Au=\sum_{i=1}^N Q_i f_i .$ 

\end{definition}

Among the class $\mathcal{M}_{loc} ( \mathbb{R}^d, \mathbb{R}^m)$ we shall consider the tempered Radon measures 
$$\mathcal M_t(\mathbb{R}^d,\mathbb R^m):=\cup_{\eta\in \mathbb{R}^+}\mathcal{M}_\eta ( \mathbb{R}^d, \mathbb{R}^m),$$
$$\mathcal{M}_\eta ( \mathbb{R}^d, \mathbb{R}^m) := \{ \mu \in \mathcal{M}_{loc} ( \mathbb{R}^d, \mathbb{R}^m); \exists C>0,\, \forall R>1,\, |\mu| (B(0, R)) \leq C R^\eta \}.$$

\medskip

Our  main result is a connection between the support and the polarisation of the singular part of a measure and its wave front set. 
\begin{theorem}\label{thWFL1} For $\mu\in\mathcal M_t(\mathbb{R}^d,\mathbb R^m)$ we have the inclusion:
\begin{equation}\label{singpolincl}
\Bigl(x,\frac{d\mu}{d|\mu|}(x)\Bigr)\overset{|\mu|_s-a.e}{\in}\Pi_{13}(WF_{L^1}(\mu)),
\end{equation}
where $\Pi_{13}$ denotes the projection with respect to the first and third variables, i.e. for $|\mu|_s$ almost all $x$, 
$$ \exists \xi \in \mathbb{S}^{d-1},\: \Bigl( x, \xi, \frac{d\mu}{d|\mu|}(x)\Bigr)\,{\in} \,WF_{L^1}(\mu).$$
\end{theorem}
\begin{rem}\label{front-onde-rem}\ 
\begin{enumerate}
\item Instead of considering temperate distributions, we could consider general distributions simply by adding a cut-off in the definition of $WF_{L^1}$ below i.e. replacing $Au$ by $A \chi u$ in ~\eqref{definiWF}. In the same vein, it is easy to see that for any $\chi \in C^\infty_c( \mathbb{R}^d)$, $\mu \in \mathcal{M}_t(\mathbb R^d,\mathbb R^m)$,
\begin{equation}\label{WFloc} WF_{L^1}( \chi \mu ) =  WF_{L^1}(  \mu ) \cap \{( x,\xi,w)\in \mathbb{R}^d\times\mathbb R^{d*}\times\mathbb R^{m*,}\, \chi (x) \neq 0\}.\end{equation}
Indeed, let $\widetilde{\chi} \in C^\infty_c( \mathbb{R}^d)$ equal to $1$ on the support of $\chi$. Then 
$$A \mu = Qf \Rightarrow \widetilde{\chi} A (\chi \mu) = \chi Q f + \widetilde{\chi} [ A, \chi] \mu,$$ 
and $\tilde \chi[A, \chi]$ is a pseudodifferential operator of order $-1$ which sends $\mathcal{M}_t $ to $L^1_{loc}$ in view of Proposition \ref{compactbis}.

 \item  We can obtain $WF_{L^1}$ by considering in Definition \ref{defWFL1} the existence of $n,N\in\mathbb N^*, A \in M_{n\times m}(\Psi^0_{\text{cl}}(\mathbb R^d))$, $Q^j\in M_{1\times N}(\Psi^0_{\text{cl}}(\mathbb R^d))$ and $f^j\in M_{N\times 1}(L^1( \mathbb{R}^d,\mathbb R))$ such that $(A u)_j =Q^j f^j$ holds for $1\leq j\leq n$, and $a_0(x, \xi)\omega\neq 0_{\mathbb R^n}$. Indeed, if $a_0 (x, \xi) \omega_0\neq 0_{\mathbb R^n}$, then there exists one line such that $(a_0 (x, \xi) \omega_0)_{j_0}\neq 0$, and choosing this  line as the $M_{1\times m}(\Psi^0_{\text{cl}}(\mathbb R^d))$ test operator allows to conclude.
 \item We can also define $WF_{L^1} $ using operators of order $k \in \mathbb{R}$ with the same order $k$ for both $A$ and $Q_i$ (by replacing $A$ by $(1- \Delta)^{-k/2} A$).
 \item In the case of scalar distributions Definition \ref{defWFL1} of $WF_{L^1}$  recovers the classical shape of a subset of $T^*\mathbb{R}^d \setminus\{0\}$ and the definition of ellipticity $a_0(x,\xi)\neq 0$ is the usual one. 
 \item The set $WF_{L^1}(u)$ together with the trivial fibers ($\omega=0_{\mathbb R^m}$) form a closed set, conical in $\xi$ and linear in $\omega$.
\item For conciseness, we shall state the results involving hypothesis as equality \eqref{definiWF}  with only one term of type $\Psi^0_{\text{cl}}\,L^1(\mathbb R^d,\mathbb R)$ in the right-hand side, but of course our results hold for finite sums of such terms. 
\item  It is very likely that our results remain true even after relaxing the smoothness assumptions on the pseudodifferential operators or even for paradifferential operators~\cite{Me}. However, for the sake of simplicity, we decided not to pursue this issue in this article.

  \end{enumerate}
\end{rem}
\subsection{Consequences}
Starting from  Theorem \ref{thWFL1} we shall deduce a series of results.  First we get a proper microlocal characterisation of the singular support of a scalar or vector measure. 
\begin{theorem}\label{corWFL1}
For $\mu\in\mathcal M_t(\mathbb{R}^d,\mathbb R^m)$ we have:
$$\Pi_{1}(WF_{L^1}(\mu)))=\varnothing\iff \mu\in L^1_{loc}(\mathbb{R}^d,\mathbb R^m),$$
and moreover
$$\Pi_{1}(WF_{L^1}(\mu)))\subseteq \mbox{supp}\, |\mu|_s, \quad |\mu|_s(\mbox{supp}\, |\mu|_s\setminus \Pi_{1}(WF_{L^1}(\mu))))=0.$$
\end{theorem}

\bigskip

Next, we obtain a series of full $L^1$ elliptic regularity result, that are the central applications of Theorem \ref{thWFL1}.
We start with the scalar case.
\begin{theorem}\label{corellL2}
Consider $\mu\in \mathcal M_t(\mathbb R^d,\mathbb R)$. Let $k\in \mathbb{R}$ and $A\in \Psi ^k_{\text{cl}}(\mathbb R^d)$ be elliptic at $x_0\in \mathbb{R}^d$ in the sense 
$$
a_0(x_0,\xi)\neq 0, \quad \forall\xi\in\mathbb S^{d-1},
$$
and let $B\in \Psi ^k_{\text{cl}}(\mathbb R^d)$. Assume that 
$$ B\mu \in \mathcal{M}_t ( \mathbb{R}^n), A\mu \in  \Psi^0_{\text{cl}}\,L^1(\mathbb R^d).$$
 Then in a neighborhood of $x_0$, 
 $$B\mu \in L^1 (\mathbb{R}^n)
 .$$ 
\end{theorem}
More generally, for vector valued measures, we have the following result.
\begin{theorem}\label{corellL1}
Consider $\mu\in \mathcal M_t(\mathbb R^d,\mathbb R^{l\times m})$. Let $k\in \mathbb{R}$ and $A\in M_{m\times m}(\Psi ^k_{\text{cl}}(\mathbb R^d))$be elliptic at $x_0\in \mathbb{R}^d$ in the sense 
$$
\ker (a_0(x_0,\xi)) =\{ 0_{\mathbb R^m}\} , \quad \forall\xi\in\mathbb S^{d-1},
$$
and let $B, C\in M_{m\times m}(\Psi ^k_{\text{cl}}(\mathbb R^d))$ such that the principal symbols of $A,B$ and $C$, 
$$ a_k(x,\xi) , b_k (x,\xi ) , c_k (x,\xi )\in M_{m\times m}$$ satisfy 
\begin{equation}\label{commutateur}
a_k(x, \xi) b_k (x, \xi) = c_k (x, \xi) a_k (x, \xi)  .
\end{equation}
 Assume that 
$$ B\mu \in \mathcal{M}_t ( \mathbb{R}^n, \mathbb{R}^m), A\mu \in  \Psi^0_{\text{cl}}\,L^1(\mathbb R^d, \mathbb{R}^m).$$
 Then in a neighborhood of $x_0$, 
 $$B\mu \in L^1 (\mathbb{R}^n, \mathbb{R}^m)
 .$$ 
\end{theorem}
\begin{remark}
A by-product of Theorem \ref{corellL1} with $A = \text{Id}$ is that counter-examples for the lack of continuity on $L^1$ of $0$-th order operators can occur only for operators $Q\in \Psi^0_{\text{cl}}$ sending an $L^1$ function outside $\mathcal M_t$. Indeed, if $Qf= \mu \in \mathcal{M}_t$ we deduce from Theorem~\ref{corellL1} with $A=B=C= \text{ Id}$ that  $\mu \in L^1_{loc}$.
\end{remark}
\begin{rem}
If $\mathcal{B}$ is scalar i.e. $\mathcal{B}=  B(x, D_x)\text{Id},$
then assumption~\eqref{commutateur} is automatically satisfied with $B=C$ which shows that Theorem~\ref{corellL1} implies Theorem~\ref{corellL2}. As a consequence, in the particular case $k \in \mathbb{N}$, $B\mu= \partial_x ^{\alpha_0}$, for some $\alpha \in \mathbb{N}^d,  |\alpha | =k$, we get (under the same ellipticity assumption on ${A}$),
$$   \partial_x ^{\alpha_0}\mu \in \mathcal{M}_t ( \mathbb{R}^n, \mathbb{R}^m), A\mu \in  \Psi^0_{\text{cl}}\,L^1(\mathbb R^d, \mathbb{R}^m) \Longrightarrow  \partial_x ^{\alpha_0}\mu \in L^1(\mathbb R^d, \mathbb{R}^m) .$$

For instance, consider $A\in M_{m\times m}(\Psi^0_{\text{cl}}(\mathbb R^d))$  elliptic for all $x\in \mathbb{R}^d$  and  $u\in BV(\mathbb{R}^d,\mathbb R^m)$, the set of $L^1$-functions with gradient a matrix-valued finite Radon measure. We recover the full $L^1$ regularity:
$$A\nabla u \in \Psi^0_{\text{cl}}\,L^1(\mathbb R^d,\mathbb R^d)\Longrightarrow u\in W^{1,1}_{\text{loc}}(\mathbb{R}^d,\mathbb R^m).$$
\end{rem}

Previously known $L^1$ elliptic regularity results assuming $u,Pu\in L^1_{loc}$, $P$ elliptic operator  of order $k$ involve loss of derivatives or requires the use of Besov spaces $B^{k,1}_\infty$, see \cite[Theorem~2.6 and Remark~2.7]{GGP}.  At the exact $L^1$ level,  the following counterexample from \cite[Example 7.5]{GM} shows that, in general some loss is anavoidable: on the disk of radius one, consider  the function 
$$u(x)=\log\log( e{|x|^{-1}})\in W^{1,1}(B_1),$$
which satisfies
$$\Delta u=-\frac1{|x|^2\log^2( e{|x|^{-1}})}\in L^1(B_1),\quad D^2 u\notin L^1(B_1).$$
 Note that this is (of course!) not in contradiction with Theorem~\ref{corellL2} as $D^2 u$ is not a  Radon measure.  Theorems \ref{corellL2} and \ref{corellL1} eliminate the loss {\em under the additional assumption that $\mu$ is a Radon measure}. To the best of our knowledge,  even for scalar equations ($m=1$), this result is new. Let us now give an example with  a fluid dynamics flavour.
\begin{example}
Let $u\in BV_{loc}(\mathbb{R}^d, \mathbb{R}^d)$. Assume that both $\divbis u$ and $\curl u $ are in $\Psi^0_{\text{cl}} ( L^1_{loc})$. Then  
$$ \nabla u\in L^1_{loc}(\mathbb{R}^d,\mathbb R^{d\times d}).$$
\end{example}
\noindent
\begin{proof}Replacing $u$ by $\chi u, \chi \in C^\infty_0$, we can assume that $u$ is compactly supported. Now we have the following equation for any derivative $\partial_iu$, which is a Radon measure by the BV assumption on $u$,
\begin{multline}
A(\partial_iu):=  \frac{1} { (1+ |D_x|^2)^{\frac 1 2}}\left(\begin{array}{cccccc}\partial_1 & \partial_2 &\partial_3 & \cdots & \cdots & \partial_d\\ \partial_2 & -\partial_1 &0 & \cdots & \cdots & 0 \\  0 & \partial_3 & -\partial_2 & 0 & \cdots & 0\\ \cdots & \cdots &\cdots & \cdots & \cdots & \cdots\\  0 & 0 & 0 & \cdots &\partial_{d} & -\partial_{d-1}\\ -\partial_d & 0 & 0 & \cdots & \cdots & \partial_1 \end{array}\right)
\partial_i u\\
=  \frac{ \partial_i} { (1+ |D_x|^2)^{\frac 1 2}}\begin{pmatrix} \divbis u \\ \curl u \end{pmatrix},
\end{multline}
so $A(\partial_iu)\in \Psi_{cl}^0 \,L^1 ( \mathbb{R}^d,\mathbb R)^{d+1},$
and a straightforward calculation shows that the constant coefficients operator $A\in M_{d+1\times d}(\Psi_{cl}^0(\mathbb R^d))$ is elliptic:
$$ \forall \xi \neq 0,\, \ker a_0(\xi) = \{0_{\mathbb R^{d+1}}\}.$$
We deduce, using the second point in Remark~\ref{front-onde-rem}  that $WF_{L^1} (\partial_i u) =\varnothing$ and then 
Theorem \ref{thWFL1} yields $ \partial_i u\in L^1_{loc}(\mathbb{R}^d,\mathbb R^{d}).$ We could also compose by $A^*$ and apply Theorem~\ref{corellL1} to the elliptic operators $A^*A$, and the choice $B =C  =\text{ Id}$.
\end{proof}
\begin{remark}\label{Radu}
In the particular case in which both $\divbis u$ and $\curl u $ are in $L^1_{loc} (\mathbb{R}^d)$ the result follows by applying Alberti's rank one theorem that ensures the existence of $a,b\in\mathbb R^d$ such that $\nabla u-a\otimes b |\nabla u|_s\in L^1$. Indeed this implies that $\divbis u-a.b |\nabla u|_s$ and $\curl u-a\wedge b |\nabla u|_s$ are in $L^1_{loc} (\mathbb{R}^d)$, so $a.b=a\wedge b=0$ thus $\nabla u\in L^1$.
\end{remark}

We now turn to an extension of De Philippis and Rindler's result~\cite[Theorem 1.1]{DePRi}. It was stated in  \cite[Remark 1.4 ]{DePRi} that the result extends easily to variable coefficients differential operators and "similar statements can be obtained if $\mu$ solves {\em some} pseudo-differential equations" by using 
$$\Lambda_{\mathcal A}(x)=\underset{|\xi|=1}{\bigcup}\ker (a(x,\xi)).$$
While the proof in \cite{DePRi} is indeed not  much perturbed by considering variable coefficients,  we shall handle here the case of {\em general} pseudodifferential operators that induce more serious issues due to their pseudo-local character.

 \begin{theorem}\label{th0}
Let $\mu\in\mathcal M(\mathbb{R}^d,\mathbb R^m)$. 
Let $A\in M_{n\times m}(\Psi^0_{\text{cl}}(\mathbb R^d))$ such that
$$A\mu\overset{\mathcal D'}{=}0_{\mathbb R^n}.$$
Then
$$
\frac{d\mu}{d|\mu|}(x)\overset{|\mu|_s-a.e}{\in}\Lambda_{A}(x). 
$$
\end{theorem}

For $A$ of order $k\geq 0$ and for $Q$ or order less than or equal to $k$, the theorem extends to the more general equation $\mathcal A\mu=Qf$ for $f\in L^1$. In the same spirit, for $Q$ or order strictly less than $k$ the result extends to the more general equation $\mathcal A\mu=Q\sigma$ with $\sigma\in\mathcal M_t$ that can be treated viewing it as $(A,-Q)(\mu,\sigma)=0$ and using the fact that the principal symbol of $(A,-Q)$ is $(a_k,0)$.

We give here the short proof of Theorem \ref{th0} as a corollary of the main theorem.
Theorem \ref{thWFL1} gives us the existence for $|\mu|_s$-almost all $x$ of at least one $\xi_x\in\mathbb R^{d*}$ such that 
$$(x,\xi_x,\frac{d\mu}{d|\mu|}(x))\in WF_{L^1}(\mu).$$
As $A\mu\in \Psi^0_{cl}L^1$, from definition \eqref{definiWFdir} of $WF_{L^1}(\mu)$ we get $\frac{d\mu}{d|\mu|}(x)\in\ker a_0(x,\xi_x)$, thus the conclusion of Theorem~\ref{th0}.

We note that Theorem \ref{th0} does not imply trivially Theorem \ref{thWFL1}. Indeed, this would be the case if one could easily get
\begin{equation}\label{diff}\mu_s(\cap_{A\mu\in\Psi^0L^1}\,\Lambda_A(x))\setminus\cup_{|\xi|=1}\cap_{A\mu\in\Psi^0L^1}\ker (a_0(x,\xi))=0.\end{equation}
Thus the difference between the two theorems is encoded in this non-trivial information which derives from Propositions 3.1-3.3.

Theorem~\ref{th0}, that allows $0$-th order pseudodifferential operators, is relevant both for the scalar or vector valued measures {\em and} for elliptic equations, as illustrated by the previous elliptic regularity results. $0$-th order pseudodifferential operators appear naturally in harmonic analysis and PDEs, let us quote just some of the most famous ones: the Hilbert and Riesz transforms ($H=\frac{\nabla}{|\nabla|}, R_i=\frac{\partial_i}{|\nabla|}$), Hodge operator, Leray projector and the pressure operator ($\mathbb P=I-\nabla\frac{\nabla}{\Delta}, p(u)=\sum R_iR_j u_iu_j$)\footnote{Notice that homogeneous $0$-th order operators are strictly speaking not pseudodifferential operator due to the singularity at the origin in the $\xi$ variable. However, our theorems still apply in this context after composing by a frequency cut-off $(1- \chi( D_x))$, $\chi \in C^\infty_0$ equal to $1$ near $0$, which eliminates this singularity, while the new operator obtained by composition with $(1- \chi(D_x))$ has the same principal symbol.}. We want to emphasize that allowing $0$-th order pseudodifferential operators is not anecdotal:  in fluid mechanics, $0$-th order operators are ubiquitous and their lack of $L^1$ (or equivalently  $L^\infty$) boundedness is responsible for many pathological behaviours (see e.g. the double exponential growth of the complexity of solutions to two dimensional Euler equations~\cite{KiSv}).

Theorem 1.1 in \cite{DePRi} implied several major consequences. Our new version can be used to relax assumptions: for instance Theorem~\ref{th0} shows that Alberti's rank one theorem for gradient of BV functions (\cite{Al}, see also \cite{DeL}) also holds for measures whose curl is a combination of first order derivatives of $L^1$ functions:
\begin{multline}
\mu\in \mathcal M_t(\mathbb R^d,\mathbb R^{l\times d}),\; \curl  \mu= (\partial_{p} \mu_{i,q}- \partial_{q} \mu_{i,p} )_{i\leq l, p\neq q\leq d}\in \Psi^1_{\text{cl}} L^1(\mathbb R^d,\mathbb R)^{d^2l}\\
\Longrightarrow \frac{d\mu}{d|\mu|}(x)\overset{|\mu|_s-a.e}{\in}\{a\otimes \xi; \: a\in\mathbb R^l,\xi\in\mathbb R^{d*}\}.
\end{multline}

The result in \cite{DePRi} has been extended in \cite{A-RDePHiRi}; it would be interesting to see whether  these results of dimensional estimates and rectifiability can be viewed and analyzed from the microlocal analysis point of view.

Finally, the microlocal perspective also allows to extend some invariance properties for constrainted measures.  We prove some propagation type results in Section~\ref{sectpropag} which show that this point of view is useful and supports the interest of defining the $L^1$- wave front set notion.

\subsection{Sketch of the proof of Theorem \ref{thWFL1}}
The proof of Theorem \ref{thWFL1} goes as follows. Arguing by contradiction we consider the set $E$ of positive $|\mu|_s$-measure on which the inclusion \eqref{singpolincl} fails. 
We first use Proposition \ref{b} to construct at any point $x\in E$ an elliptic operator $B_x$ at $\bigl(x,\frac{d\mu}{d|\mu|}(x)\bigr)$ that smoothens $\mu$ in the $L^1$-sense of \eqref{WFL1}. Then we use the equation in \eqref{WFL1} as a starting point to implement an elliptic regularity strategy {\em \`a la De Philippis-Rindler~\cite{DePRi}} relying on a precise description of tangent measures near $\mu_s$-generic points. The main difficulties in this approach are the following:
\begin{itemize}
\item We need some regularity properties of the $L^1$ functions $f_i$ in~\eqref{WFL1} with respect to the $\mu_s$ measure. These properties are apriori valid only $|\mu|_s$-a.e., while the functions themselves depend on $x$; this issue is solved in Proposition \ref{eqcovbis}.
\item We need to handle the non locality of pseudodifferential operators, which are only {\em pseudo-local}, and this requires some uniform temperance properties of the scaled measures in the definition of the tangent measures.
\item Finally we need to handle the limiting case of $0$-th order operators with no gain of regularity. \\
\end{itemize}


The article is structured as follows.  In Section~\ref{sec.ex} we give some natural examples illustrating the relevance of $WF_{L^1}$. 
Section \ref{sectthWFL1} contains the proof of the main result, namely Theorem~\ref{thWFL1}. In Section \ref{sectcorWFl1} we give the short proofs that Theorem~\ref{thWFL1} implies Theorems \ref{corWFL1}, \ref{corellL2} and \ref{corellL1}. In Section \ref{sectpropag} we consider the propagation of singularities results. 
In Section \ref{sec.pseudo} are gathered the definitions and basic results on pseudodifferential operators that we use throughout the article. The last section contains general results of temperance properties for measures, obtained from the construction by Preiss in \cite{Pr}, which are important when dealing with pseudodifferential operators.\\

{\bf{Aknowledgements: }} We are grateful to Guido de Philippis for the lectures he gave in Paris in 2017 that allowed us to learn about these topics. We would like to thank Fr\'ed\'eric Bernicot for enlightenments about weak $L^1$ estimates for pseudodifferential operators, and Radu Ignat for pointing us the Remark \ref{Radu}. We would also like to thank the referee for the improving suggestions. \par
We are grateful to the Institut Universitaire de France for the ideal research conditions offered by their memberships. The first author was also partially supported by the French ANR project SingFlows ANR-18-CE40-0027,  while the second author was also partially supported by ANR project ISDEEC  ANR-16-CE40-0013.



\section{Examples}\label{sec.ex}
In this section we give some examples of measures for which we can describe the $L^1$ wave front, $WF_{L^1}$. We illustrate  the relevance of both the cotangent variable $\xi$ and the polarisation variable $\omega$ by providing examples for which the dependance with respect to these variables is non trivial. 
\begin{prop}
\begin{enumerate}
\item \label{0}
Contrarily to the usual wave fronts defined in microlocal analysis for distributions, in general, the natural implication
\begin{equation}\label{regu}
WF_{L^1}(u) = \varnothing \Rightarrow u \in L^1_{loc},
\end{equation}
is {\em not true}. There exists distributions $\mu \in \mathcal{S}'(\mathbb{R}^d)$ which are not in $L^1_{loc}$ but such that 
$$ WF_{L^1}(\mu)= \varnothing .$$

\item   \label{1} Let $\mu \in L^1_{loc}(\mathbb R^d,\mathbb R^m)$. Then $WF_{L^1}( \mu)=\varnothing$ (and according to Theorem~\ref{corWFL1} the converse is true if $\mu \in \mathcal{M}_t(\mathbb R^d,\mathbb R^m)$).
\item  \label{2}Let $\mu= (\mu_1, \mu_2) , \mu_1 \in L^1_{loc} ( \mathbb{R}^d; \mathbb{R}^{m_1}) $ and  $\mu_2\in \mathcal{M}_t( \mathbb{R}^d; \mathbb{R}^{m_2})$.
Then 
$$ (x, \xi, \omega) \in WF_{L^1} (\mu) \Leftrightarrow  (x,\xi,\omega_2) \in  WF_{L^1}(\mu_2) \text{ and } \omega= (0_{\mathbb R^{m_1}}, \omega_2).$$
\item  \label{3} Let $x_0\in\mathbb R^d, \omega_0 \in \mathbb{R}^{m*}$ and $\mu=\omega_0 \delta_{x_0}$. Then 
$$ WF_{L^1} (\mu) = \{ (x_0,\xi, \omega); \xi\in\mathbb R^{d*}, \omega \in \mathbb{R^*}\omega_0\}$$
\item  \label{4}Let $ \Sigma $ a smooth $d_1$ dimensional embedded submanifold of $\mathbb{R}^d$, and $\mu$ the scalar surface measure on $\Sigma$. Then 
$ WF_{L^1} (\mu)$ is the set of points $(x, \xi) \in \mathbb R^d\times\mathbb R^{d*} $ such that $x\in \Sigma$ and $\xi$ is (co)-normal to $\Sigma$: 
$$ \forall X \in T_x \Sigma, \langle \xi, X \rangle =0.$$
 \item  \label{5}Let $(e_p)_{p=1}^m$ be the canonical basis of $\mathbb{R}^m$. We consider a sequence $\{x_n\}_{n\in\mathbb N}$ of pairwise disjoint points such that  $x_n \overset{n\rightarrow + \infty}{\longrightarrow}  x_0$ and 
 $$\mu = \sum_{n\in\mathbb N} 2^{-n} \delta_{x_n} \omega_n,\qquad  \omega_n =  e_p, \text{ if } n\overset{}\equiv p \,  (\text{mod } m).$$
 Then 
 $$ WF(\mu) =\bigcup_{n\geq 1} \{x_n\} \times\mathbb R^{d*} \times\mathbb{R^*} \omega_n \,\bigcup\, \{x_0\} \times\mathbb R^{d*} \times\mathbb R^{m^*}.$$
 \end{enumerate}
\end{prop}
\begin{proof} 
 To prove ~\eqref{0},  recall that general pseudodifferential operators of order $0$ are not bounded on $L^1$.  Taking for example the function $u= D^2 w$ and $f$ from~\eqref{contre-ex}, we have  $f\in L^1_{comp}$, $Q= D^2 (- \Delta+1)^{-1}\in \Psi^0_{\text{cl}}$ such that $Qf = u \notin L^1$. On one hand this implies by definition that $WF_{L^1}(u) =0_{\text{pol}}$, while  $u \notin L^1_{loc}$ (because $u$ is a compactly supported distribution  but is not in $L^1$). 
However, from Theorem~\ref{corWFL1}, if we assume in addition  $u \in \mathcal{M}_t$, then  property~\eqref{regu} is true. 

Example~\eqref{1} is trivial, as in the definition of $WF_{L^1}$,  we can choose for any $x_0\in\mathbb R^d$ and $\omega_0 \in \mathbb{R}^{m*}$, the multiplication operator  $ A= \chi(x)\, ^t \omega_0 $ with $\chi\neq 0$ in a neighborhood of~$x_0$. 

To prove Example~\eqref{2}, we consider first $(x,\xi,\omega)$ with $\omega=(\omega_1,\omega_2)$ such that $\omega_1\neq 0_{\mathbb R^{m_1}}$. Assume for example that the first coordinate of $\omega_1$ does not vanish.
Let 
$$ A= \begin{pmatrix} 1, 0, \dots, 0\end{pmatrix} \in M_{1\times m} ( \mathbb{R}), $$
and $f$ be the first component of the vector valued function $\mu_1$. Then we have 

$$ A \mu = f\in L^1_{\text{loc}}(\mathbb R^d,\mathbb R),$$ 
and the constant operator $A$ is elliptic at $(x, \xi, \omega)$ because the first coordinate of $\omega_1$ does not vanish. 
We deduce that $(x, \xi, \omega)\notin WF_{L^1} (\mu).$ Let now  $(x, \xi,0_{\mathbb R^{m_1}},\omega_2)\notin WF_{L^1} (\mu).$ By definition, there exists a $A, Q_j, f_j $ as in Definition~\ref{defWFL1}  such that $A$ is elliptic at $(x, \xi, 0,\omega_2)$ and 
$$ A \mu= \sum_j Q_j f_j.$$ 
which implies by linearity
$$ A  \begin{pmatrix} 0 \\ \mu_2 \end{pmatrix}= \sum_j Q_j f_j -A \begin{pmatrix} \mu_1 \\ 0 \end{pmatrix}$$ 
or equivalently 
$$ A_2  \mu_2 = \sum_j Q_j f_j +A_1 \mu_1,  \qquad A = (A_1,A_2).$$ 
Since by assumption $\mu_1\in L^1$, to conclude that $(x, \xi, \omega_2) \notin WF_{L^1}(\mu_2)$. It remains to check that if $A$ is elliptic at $(x, \xi, 0_{\mathbb R^{m_1}}, \omega_2)$ then $A_2$  is elliptic at $(x, \xi, \omega_2)$, which is clear from 
$$ a(x,\xi) \begin{pmatrix} 0\\ \omega_2\end{pmatrix} =a_2 (x, \xi) (\omega_2).$$ 
Conversely, let now  $(x, \xi, \omega_2)\notin WF_{L^1} (\mu_2).$ By definition, there exists a $A_2, Q_j, f_j $ as in Definition~\ref{defWFL1}  such that $A_2$ is elliptic at $(x, \xi, \omega_2)$ and 
$$A_2 \mu_2= \sum_j Q_j f_j.$$ 
With $A=(0_{\mathbb R^{m_1}},A_2)$ we have for any $\omega_1\in\mathbb R^{m_1}$:
$$A \begin{pmatrix} \mu_1 \\ \mu_2 \end{pmatrix} = A_2 \mu_2,\quad a(x,\xi)(\omega_1,\omega_2)=a_2(x,\xi)\omega_2\neq 0,$$
and we obtain that $(x, \xi, \omega_1,\omega_2)\notin WF_{L^1} (\mu).$

 For example~\eqref{3}, it is clear from Example ~\eqref{1} that if $(x,\xi, \omega) \in WF_{L^1}(\omega_0\delta_{x_0})$ then $x=x_0$. Let us assume that $\omega \notin \mathbb{R} \omega_0$. Choosing for $A$ the multiplication by a constant $1\times m$ vector orthogonal to $\omega_0$ but not to $\omega$, we get $A\omega_0\delta_{x_0}=0$ and $a(x_0,\xi)\omega\neq 0$ so $(x_0, \xi, \omega) \notin WF_{L^1}(\omega_0\delta_{x_0}), \forall \xi\in\mathbb R^{d*}$. As a consequence we have 
$$ WF_{L^1} (\omega_0\delta_{x_0}) \subseteq \{ (x_0,\xi, \omega); \xi\in\mathbb R^{d*}, \omega \in \mathbb{R^*}\omega_0\}.$$ 
The equality follows from the fact that $WF_{L^1}(\omega_0\delta_{x_0}) \neq \varnothing$ and $WF_{L^1}(\omega_0\delta_{x_0})$ is invariant under rotations in the $\xi$ variables, as so does $\hat \delta_{x_0}$.

For example~\eqref{4}, we perform a change of variables and are reduced to the case when 
$$ \Sigma = \{ (x_1, 0_{\mathbb R^{d_2}}) ; x_1\in\mathbb R^{d_1}\},  \quad \mu= g(x_1)  \delta_{x_2 = 0_{\mathbb R^{d_2}}},\, g\in C^\infty(\mathbb{R}^{d_1},\mathbb R),$$
where the function $g$ is positive. Let us apply a pseudodifferential operator $\chi(D_x)$ where $\chi$ is supported in a conic neighborhood of a point $\xi=(\xi_1, \xi_2)$ with $\xi_1\neq 0_{\mathbb R^{d_1}}$, small enough so that on the support of $\chi$ we have $|\xi| \leq C |\xi_1|$. We have 
$$ \chi(D_x) \mu = \frac{1} {(2\pi)^d} \int e^{i(x_1- y_1) \cdot \xi_1 + x_2 \cdot \xi_2 } \chi(\xi) g(y_1)dy_1 d\xi,$$
and integrations by parts using $\frac{\partial}{\partial {y_1}} e^{-iy_1\xi_1} = -i \xi_1  e^{-iy_1\xi_1}$ gain arbitrary inverse powers of $|\xi_1|$ (hence of $|\xi|$) which shows that $\chi(D_x) \mu \in C^\infty$. We deduce the inclusion
$$ WF_{L^1} ( \mu) \subset \{ (x_1, 0_{\mathbb R^{d_2}}, 0_{\mathbb R^{d_1}},\xi_2) ; x_1\in\mathbb R^{d_1}, \xi_2\in \mathbb R^{d_2}\}.$$
To get the equality, we first remark that it is enough to study the case above when $g= 1$, because we can multiply all pseudodifferential operators by the smooth function $g(x_1) ^{-1}$ to the right. According to Theorem~\ref{corWFL1},  $\Pi_1 WF_{L^1}$ must have a non trivial intersection with the set $\{ (x_1,0_{\mathbb R^{d_2}}); x_1\in\mathbb R^{d_1}\}$ because $\mu$ is not $L^1$ near this point. We get the equality simply because in that particular case $g_1= 1$, $WF_{L^1}$ is clearly invariant by rotations of the variable $\xi_2$ as so does $\delta_{0_{\mathbb R^{d_2}}}$.

Let us now turn to Example~\eqref{5}. Since the points $x_n$ are pairwise disjoint and converging to $x_0$, all the points $x_n, n\geq1$ are isolated in the sequence and localising $\mu$ near each point $x_n$ (i.e. considering $\mu_n= \chi_n \mu_n$, with $\chi_n=1$ near $x_n$ and $\chi_n=0$ in a neighborhood of $x_k, k\neq n$),  we get, according to Example~\eqref{3} that 
$$ WF(\mu)  \cap \{ (x,\xi, \omega); x\neq x_0\} =  \bigcup_{n\geq 1}\{x_n\}\times\mathbb R^{d*} \times \mathbb{R^*} \omega_n.$$
To conclude, it remains to prove that 
$$ WF(\mu) \cap \{ (x,\xi, \omega); x= x_0\} =\{x_0\}\times\mathbb R^{d*}  \times \mathbb{R}^{m*}.$$ 
By linearity with respect to the $\omega$ variable, it is actually enough to prove that 
$$ (x_0, \xi, e_p) \in WF_{L^1}(\mu), \forall \xi\in\mathbb R^{d*}, 1\leq p\leq m.$$
We argue by contradiction. Assume there exists $\xi_0$ and $p$ such that $ (x_0, \xi_0, e_p) \notin WF_{L^1}(\mu).$ Then there exists $A,Q_j\in \Psi^0_{\text{cl}}$ matrices of pseudodifferential operators , $f_j\in L^1$ such that 
$$A \mu = \sum_j Q_jf_j, \quad a_0 (x_0,\xi_0) e_p \neq 0.$$ 
 We deduce with $\chi_n$ as above,
 $$ A \chi_n\mu= \chi_n Q f + [A, \chi_n]  \mu= Q_n f + g, g \in L^1.$$
 But for $n$ sufficiently large, $x_n$ is arbitrarily close to $x_0$,  we get 
 $$ a(x_n, \xi_0)  e_p \neq 0.$$ 
 We deduce that 
 $$(x_n, \xi_0,  e_p) \notin WF_{L^1}(\chi_n  \mu) = WF_{L^1}(2^{-n} \omega_n\delta_{x_n} )= \{x_n\}\times\mathbb R^{d*} \times \mathbb{R^*} \omega_n,
 $$ which is a contradiction if we choose $n=p \,(\text{mod } m)$ so that $\omega_n = e_p$.
\end{proof}
 In example ~\eqref{5}, we have $\mu = \mu_s$ and 
$$ |\mu|_s = \sum_{n\in\mathbb N} 2^{-n} \delta_{x_n}, \qquad \frac{d\mu}{d|\mu|} (x)= \sum_{n\in\mathbb N} 1_{x=x_n} \omega_n.$$
As a consequence, this example shows that in Theorem~\ref{thWFL1}, the inclusion is, in general {\em not} an equality. 
This is an indication that the wave front we define in the present work, though sufficient for the applications of our present work, might  need to be refined for further applications.

\section{Proof of Theorem \ref{thWFL1}}\label{sectthWFL1} 
Let us recall that all the results are local. Indeed, following the argument in Remark~\ref{front-onde-rem}, we have
$$A \mu = Qf \Rightarrow \widetilde{\chi} A (\chi \mu) = \chi Q f + \widetilde{\chi} [ A, \chi] \mu,$$ 
and  $\widetilde{\chi} [ A, \chi] \mu \in L^1$. Using \eqref{WFloc} and its explanation, we shall assume without loss of generality that $\mu$ is compactly supported \footnote{If one wants to prove directly Theorem \ref{th0}, without using Theorem \ref{thWFL1} but starting from \S 3.3, then the localisation of the measure induce only a commutator term in the right-hand-side, that is of the same type as the initial constraint in view of the explanations of \eqref{WFloc}.}.
We argue by contradiction, and we suppose that the conclusion of Theorem \ref{thWFL1} does not hold:
$$|\mu|_s(E)>0,$$
where
$$
E:=\{x;\;\bigl(x,\frac{d\mu}{d|\mu|}(x)\bigr)\in(\Pi_{13}WF_{L^1}(\mu))^{\mathsf{c}}\}.
$$
\smallskip

We start by proving a general property of $WF_{L^1}$ that will be the starting point of our approach.
\subsection{An elliptic characterisation of  \texorpdfstring{$\Pi_{13}(WF_{L^1})$}{the L1 wave front}}\label{sectb}

In the following Proposition we give a characterisation of the elements in $ (\Pi_{13}(WF_{L^1}(u))^c$ for $u\in\mathcal {S}'(\mathbb{R}^d,\mathbb R^m)$.
\begin{prop}\label{b}  For $u\in\mathcal {S}'(\mathbb{R}^d,\mathbb R^m)$ the set $ (\Pi_{13}(WF_{L^1}(u))^c$ coincides with 
 \begin{multline}\label{WFL1}
\Bigl\{(x,\omega)\in\mathbb{R}^d\times\mathbb R^{m*};\:\exists B\in M_{1\times m}(\Psi^0_{\text{cl}}(\mathbb R^d))\mbox{ elliptic at }(x,\omega),
 \\
  N \in\mathbb N^*, Q_i\in \Psi^0_{\text{cl}}(\mathbb R^d),\,f_i\in L^1(\mathbb{R}^d,\mathbb R), \forall 1\leq i\leq N ,\; Bu =\sum_{i=1}^NQ_i f_i \Bigr\}.
  \end{multline}
 Here ellipticity at $(x,\omega)$ means 
\begin{equation}\label{elldef}
b_{0}(x,\xi)\omega\neq 0, \quad \forall\xi\in\mathbb S^{d-1}.
\end{equation}
\end{prop}

\begin{proof}[Proof of Proposition \ref{b}]
Let $(x,\omega)\in (\Pi_{13}(WF_{L^1}(u)))^c$. Then by definition for all $\xi\in\mathbb S^{d-1}$ we have the existence of $N_\xi\in\mathbb N^*, B_\xi\in M_{1\times m}(\Psi^0_{\text{cl}}), Q_{\xi,i} \in \Psi^0_{\text{cl}}(\mathbb R^d)$ and $f_{\xi,i}\in L^1(\mathbb{R}^d,\mathbb R), \forall 1\leq i\leq N_\xi$ such that $B_\xi u=\sum_{i=1}^{N_\xi}Q_{\xi,i} f_{\xi,i}$, with a principal symbol $b_{0,\xi}$ satisfying
$$b_{0,\xi} (x,\xi)\omega\neq 0.$$
By continuity we can find a conical neighborhood $V_\xi\in\mathcal V_c(\xi)$ such that
$$b_{0,\xi}(x,\eta)\omega\neq 0,\quad \forall \eta\in V_\xi.$$
and moreover $\delta_\xi\in\{\pm 1\}$ such that
\begin{equation}\label{ellconstr}\delta_\xi b_{0,\xi}(x,\eta)\omega>0,\quad \forall \eta\in V_\xi.\end{equation}
Let $\chi_\xi$ be a non-negative smooth cut-off with support in $V_\xi$, valued one on  a conical neighborhood $\tilde V_\xi\in\mathcal V_c(\xi)$ with $\tilde V_\xi\subset V_\xi$. 
As $\cup_{\xi\in \mathbb S^{d-1}}\tilde V_\xi$ is a covering of $\mathbb S^{d-1}$ we can extract a finite covering 
\begin{equation}\label{coverS}\exists L\in\mathbb N^*,\quad \mathbb S^{d-1}\subset \cup_{l=1}^L\tilde V_{\xi_l}.\end{equation}
Then we define $B:= \sum_{l=1}^L Op(\chi_{\xi_l}\delta_{\xi_l}) B_{\xi_l}\in M_{1\times m}(\Psi^0_{\text{cl}}(\mathbb R^d))$ with principal symbol
$$b_{0}(x,\xi):=\sum_{l=1}^L \chi_{\xi_l}\delta_{\xi_l}b_{0,\xi_l}(x,\xi),$$
satisfying
$$Bu=\sum_{l\in L}\sum_{i=1}^{N_{\xi_l}}Op(\chi_{\xi_l}\delta_{\xi_l})Q_{\xi_l,i} f_{\xi_l,i}.$$ 
In particular the right-hand-side is a finite sum of terms $Q f$ with $Q\in \Psi^0_{\text{cl}}(\mathbb{R}^d)$ and $f\in L^1(\mathbb{R}^d,\mathbb R)$.  We consider now
$$b_{0}(x,\xi)\omega=\sum_{l=1}^L\chi_{\xi_l}\delta_{\xi_l}b_{0,\xi_l}(x,\xi)\omega.$$
On one hand from \eqref{ellconstr} we deduce that all terms in the sum are non negative  for all $\xi\in\mathbb S^{d-1}$. On the other hand from \eqref{coverS} there exists at least a set $\tilde V_{\xi_j}$ such that $\xi\in \tilde V_{\xi_j}$, thus using again \eqref{ellconstr} we get \eqref{elldef}.\\
\end{proof}

\subsection{Preliminaries} We start with a lemma gathering several general measure properties.

\begin{lemma}\label{eqcov}  There exists a set $\mathcal N_0$ with $|\mu|_s (\mathcal N_0) =0$ such that for all $x \in \mathcal N_0^c$, there exists a sequence $r_j\rightarrow 0$ such that the following three assertions hold.
\begin{equation*} \tag{i} \label{i)}\underset{r\rightarrow 0}{\lim}\frac{\mathcal L^d (B(x,r))}{|\mu|_s (B(x,r))}=0,\; \underset{r\rightarrow 0}{\lim}\frac{|\mu|_a (B(x,r))}{|\mu|_s (B(x,r))}=0,\; \underset{r\rightarrow 0}{\lim}\avint_{B(x,r)}\left|\frac{d\mu}{d|\mu|}(y)-\frac{d\mu}{d|\mu|}(x)\right|d|\mu|_s=0,\end{equation*}
There exists $C,c>0$ such that the family of measures \footnote{by $ T^{x, r_j}$ we denote the dilations $\tilde x\in \mathbb{R}^d \mapsto x + r_j \tilde x \in \mathbb{R}^d,$ and by  $ T^{x, r_j}_\sharp\sigma $ we denote the push-forward of the measure $\sigma$ given by $T^{x, r_j}_\sharp\sigma(A)=\sigma(x+r_j A)$.}
$$\mu_j :=\frac{(T^{x,r_{j}})_\sharp \mu}{|\mu _s| (B(x,r_{j}))}$$ 
satisfies the temperance uniform in $j$ bounds:
\begin{equation*} \tag{ii} \label{ii)}    |\mu_j|( B(x, R)) = \frac{ |\mu| (B(x, R r_j)) } { |\mu|_s (B(x, r_j))}\leq C \max\{1,R^{d+\frac 12}\},
\end{equation*}
 We can pass to the limit
\begin{equation*} \tag{iii} \label{iii)}\mu_j\underset{j\rightarrow\infty}{\overset{*}{\rightharpoonup}}\tilde\nu\in Tan(x,\mu),\quad \tilde\nu=\frac{d\mu}{d|\mu|}(x)\nu, \quad \nu\in\,Tan(x,|\mu|),\end{equation*}
and the restriction of $\nu$ to $B_\frac 12$ is not trivial. 
\end{lemma}
Assertion (i) is based on classical characterisation of singular points of a measure (see e.g.~\cite[Theorem 2.12 (3)]{M}, ~\cite[Corollary 1.7.1]{EvGa}). Assertion (ii) is shown in Proposition~\ref{Preiss}. Assertion (iii) is given by the classical properties of tangent measures (see e.g. \cite[Section 2.3]{Ri}) and the lower bound~\eqref{dyadicboundbis}.

Next we are going to prove the following property on the set $E\setminus \mathcal N_0$.

\begin{prop} \label{eqcovbis}There exists a set $\tilde E\subseteq E\setminus \mathcal N_0$ with $|\mu|_s(E \setminus \tilde E)=0$ such that 
for all $x\in\tilde E$, there exists $B\in M_{1\times m}(\Psi^0_{\text{cl}}(\mathbb{R}^d)),N \in\mathbb N^*,   Q_i\in \Psi^0_{\text{cl}}(\mathbb{R}^d), f_i \in L^1(\mathbb{R}^d,\mathbb R), \forall 1\leq i\leq N$, satisfying the identity
$$ Bu=\sum_{i=1}^{N}Q_if_i,$$
and the elliptic condition at $\bigl(x,\frac{d\mu}{d|\mu|}(x)\bigr)$:
$$
b_{0}(x,\xi)\frac{d\mu}{d|\mu|}(x)\neq 0, \quad \forall\xi\in\mathbb S^{d-1}.
$$
Moreover, for the sequence $(r_j)$ constructed in Lemma~\ref{eqcov} we have the following properties:
\begin{equation}
\begin{gathered} \tag{iv} \label{iv)}\forall i\in\{1,..,N_x\},\quad  \underset{r\rightarrow 0}{\lim}\frac{|f_i\mathcal L^d  (B(x,r))|}{|\mu|_s  (B(x,r))}=0, \\
\exists C_{x,i}>0,\: \frac{|f_i\mathcal L^d  (B(x,Rr_j))|}{|\mu|_s  (B(x,r_j))} \leq C_{x,i} \max\{1,R^{d+\frac 12}\}.
\end{gathered}
\end{equation}
\end{prop}
\begin{proof}
As $|\mu|_s$ is a finite Radon measure we have $|\mu|_s(E\setminus \mathcal N_0)= \sup \{ \mu_s (K); K \text{ compact} \subset E\setminus \mathcal N_0\}$. Thus there exists a compact set $K\subseteq E \mathcal N_0$ such that $|\mu|_s(K)>\frac 1 2|\mu|_s(E) >0$.

For any $x\in K$ we have $x\in E$, so $\bigl(x,\frac{d\mu}{d|\mu|}(x)\bigr)\in(\Pi_{13}WF_{L^1}(\mu))^{\mathsf{c}}$. Thus we can apply Proposition \ref{b} to get the existence of 
$B_{x}\in M_{1\times m}(\Psi^0_{\text{cl}}(\mathbb{R}^d)),$ 
 $N\in\mathbb N^*, Q_{x,i}\in M_{m\times 1}(\Psi^0_{\text{cl}}(\mathbb{R}^d)), f_{x,i}\in L^1(\mathbb{R}^d,\mathbb R), \forall 1\leq i\leq N$, such that
$$B_{x}\mu =\sum_{i=1}^{N_x}Q_{x,i}f_{x,i},$$
with the elliptic property at $\bigl(x,\frac{d\mu}{d|\mu|}(x)\bigr)$:
$$b_{0,x}(x,\xi)\frac{d\mu}{d|\mu|}(x)\neq 0, \quad \forall\xi\in\mathbb S^{d-1}.$$
This implies the existence of $r_{x,1}>0$ and $c_x>0$ such that 
$$|b_{0,x}(y,\xi)\frac{d\mu}{d|\mu|}(x)|>c_x>0,\quad \forall y\in B(x, r_{x,1}),\quad \forall\xi\in\mathbb S^{d-1},$$
In particular there exists $\delta_x>0$ such that
\begin{equation}\label{delta}
b_{0,x}(y,\xi)\omega\neq 0,\quad \forall y\in B(x, r_{x,1}),\quad \forall\xi\in\mathbb S^{d-1},\quad\forall \omega\in\mathbb S^{m-1},\,\left\|\omega-\frac{d\mu}{d|\mu|}(x)\right\|\leq \delta_{x}.\end{equation}
From Lemma~\ref{eqcov} for all $x \in \mathcal N_0^c$, thus for all $x\in K$,
\begin{equation}\label{polmeans} 
\underset{r\rightarrow 0}{\lim}\avint_{B(x,r)}\left|\frac{d\mu}{d|\mu|}(y)-\frac{d\mu}{d|\mu|}(x)\right|d|\mu|_s=0,
\end{equation}
with the average integral well-defined, meaning $|\mu|_s(B(x,r))>0,\, \forall r,\,\forall x\in K$. Thus for all $x\in K$ we get the existence of $r_{x,2}$ such that for all $r\leq r_{x,2}$
\begin{equation}\label{r2x}\avint_{B(x,r)}\left|\frac{d\mu}{d|\mu|}(y)-\frac{d\mu}{d|\mu|}(x)\right|d|\mu|_s\leq \frac{\delta_x} 2.\end{equation}
For all $x\in K$ and $r\leq r_{x,2}$ we have $|\mu|_s(B(x,r))>0$ and the set defined by 
$$F_{x,r}:=\{y\in B(x,r),\, \left\|\frac{d\mu}{d|\mu|}(y)-\frac{d\mu}{d|\mu|}(x)\right\|\leq \delta_x\},$$
satisfies, in view of \eqref{r2x}, 
\begin{equation}\label{posmesF}F_{x,r}\subseteq B(x,r),\quad |\mu|_s(F_{x,r})>\frac{ |\mu|_s(B(x,r))} 2.\end{equation}

We denote by $\mathcal N_1$ the set of null $|\mu|_s-$measure of the zero-density points w.r.t. $|\mu|_s$ of the set $K$, thus
\begin{equation}\label{posmesK}\forall x\in K\setminus \mathcal N_1,\quad \exists r_{x,3}, \quad \forall r\leq r_{x,3}\quad\frac 34 |\mu|_s( B(x,r))\leq |\mu|_s(K\cap B(x,r)).\end{equation}

The closed balls $\overline{B}(x,\frac{r_x}2), x\in K \setminus \mathcal N_1$ with $r_x=\min\{r_{x,1},r_{x,2},r_{x,3}\}$ form a covering of the bounded set $K\setminus \mathcal N_1$, so by Besicovitch's covering theorem (\cite[Theorem 2.7]{M}) we can extract a countable covering: there exists $\{x_n\}_{n\in\mathbb N}\subseteq K\setminus \mathcal N_1$ such that 
$$K\setminus \mathcal N_1\subseteq\cup_{n\in\mathbb N}\overline{B}(x_n,\frac{r_{x_n}}2)\subseteq\cup_{n\in\mathbb N}B(x_n,r_{x_n}).$$

For all $n\in\mathbb N$ and $1\leq i\leq N_{x_n}$ there exists a set $\mathcal N_{n,i}$ of null $|\mu|_s$-measure such that 
$$\underset{r\rightarrow 0}{\lim}\frac{|f_{x_n,i}\mathcal L^d  (B(x,r))|}{|\mu|_s  (B(x,r))}=0,\:\forall x\in (\mathcal N_{n,i})^c.$$
This is due to the fact that $|\mu|_s\perp f_{x_n,i}\mathcal L^d$ and Theorem 2.12 (3) in \cite{M}.

It follows from \eqref{delta} that
\begin{equation}\label{ellshift}
\begin{gathered}
\forall x\in K\setminus \mathcal N_1,\quad\exists n\in\mathbb N, \quad x\in B(x_n,r_{x_n}),\quad \mbox{ and if } \left\|\frac{d\mu}{d|\mu|}(x)-\frac{d\mu}{d|\mu|}(x_n)\right\|\leq \delta_{x_n},\\
\mbox { then }\quad b_{0,x_n}(x,\xi)\frac{d\mu}{d|\mu|}(x)\neq 0\quad  \forall\xi\in\mathbb S^{d-1}.
\end{gathered}
\end{equation}

Now we define the sets
$$\tilde K:=K\setminus (\mathcal N_1\cup_{n\in\mathbb N}\cup_{i=1}^{N_{x_n}}\mathcal N_{x_n,i}),$$
and
$$\tilde E:=(\cup_{n\in\mathbb N} F_{x_n,r_{x_n}})\cap \tilde K\subseteq E\setminus \mathcal N_0.$$ 
In view of \eqref{posmesF}-\eqref{posmesK} we have
\begin{multline}
|\mu|_s(F_{x_n,r_{x_n}}\cap \tilde K) = |\mu|_s(F_{x_n,r_{x_n}}\cap  K) \\
\geq |\mu|_s(F_{x_n,r_{x_n}}) - |\mu|_s(B(x_n, r_{x_n}) \cap K^c) \geq (\frac 1 2 -  \frac 14) |\mu|_s(B(x_n,r_{x_n})),
\end{multline}
and since $|\mu|_s( K) >0$ we get at least one $n\in\mathbb N$ such that $|\mu|_s(B(x_n,r_{x_n}))>0$ so
$$|\mu|_s(\tilde E)>0.$$ 
Then \eqref{ellshift} and the definitions of $F_{x,r}$ and $\tilde E$ ensure us that
$$\forall x\in\tilde E,\quad\exists n\in\mathbb N,\quad b_{0,x_n}(x,\xi)\frac{d\mu}{d|\mu|}(x)\neq 0,\quad \quad \forall\xi\in\mathbb S^{d-1}.$$
Therefore for $x\in\tilde E$ we define  $B:=B_{x_n}$, which have the elliptic property in the statement of the Lemma, we define $N:=N_{x_n}$, and we have the first part of the Lemma:
$$B u=\sum_{i=1}^{N}Q_{x_n,i} f_{x_n,i}.$$

Moreover, $f_{x_n,i}$ satisfies the first assertion of (iv) as by definition $\tilde E\subset \mathcal N_{x_n,i}^c$. Then it follows that
 $$\exists R_{x,x_n,i},\; \forall r\in (0, R_{x,x_n,i}),\: \frac{|f_{x_n,i}\mathcal L^d  (B(x,r))|}{|\mu|_s  (B(x,r))}\leq 1.$$
For $R r_j \leq R_{x,x_n,i}$ we deduce by combining also with Lemma \ref{eqcov} (ii) that
\begin{multline}
\frac{|f_{x_n,i}\mathcal L^d (B(x,R r_j))|}{|\mu|_s(B(x,r_j))}= \frac{|f_{x_n,i}\mathcal L^d (B(x,R r_j))|}{|\mu|_s(B(x,Rr_j))}\frac{|\mu|_s (B(x,R r_j))}{|\mu|_s(B(x,r_j))}\\
\leq  \frac{|\mu|_s (B(x,R r_j))}{|\mu|_s(B(x,r_j))}\leq C \max\{1,R^{d+\frac 12}\}.
\end{multline}
On the other hand, by using Lemma~\ref{eqcov} (i) we have for $R r_j \geq R_{x,x_n,i}$
$$
\frac{|f_{x_n,i}\mathcal L^d (B(x,R r_j))|}{|\mu|_s(B(x,r_j))}\leq  C\frac{\| f_{x_n,i}\|_{L^1}} {r_j^{-d}} \leq \frac{C'}{R_{x,x_n,i}^d} R^{d}.
$$
Therefore $f_{x_n,i}$ satisfies also the last assertion of (iv).

Summarizing, we have shown that we can find  $\tilde E\subset E\setminus \mathcal N_0$ of  positive $|\mu|_s-$measure, on which all the conclusion of the Lemma are satisfied, excepted the fact that $|\mu_s|(E\setminus \tilde{E}) =0$.
The argument above is valid also starting from any subset of $E\setminus \mathcal N_0$ of  positive $|\mu|_s-$measure. Thus the conclusion of the Lemma are satisfied $|\mu|_s-$a.e. on $E\setminus \mathcal N_0$. Indeed, let us we denote by $E_1$ the subset of $E\setminus \mathcal N_0$ of the points for which the conclusion of the Lemma are not satisfied. If $|\mu_s| (E_1)>0$ then the construction above would give the existence of a subset $\tilde{E}_1\subset E_1$ of strictly positive $|\mu_s|-$measure, thus non-empty, on which the conclusions of the Lemma are both true and false on $\widetilde{E}_1$. Therefore $|\mu_s|(E_1)=0$ and the full conclusion of the Lemma follows.
\end{proof}
\subsection{The contradiction argument}
We go back to the contradiction argument. From now on, our proof is inspired from the approach  in ~\cite{DePRi}, using the elliptic type operator family from Proposition \ref{eqcovbis} instead of an operator constraining the measure to vanish.

We thus apply Proposition \ref{eqcovbis} and as $|\mu|_s(\tilde E)>0$ we can choose now  $x_0\in\tilde E$ so that properties in Lemma \ref{eqcov} and Proposition-\ref{eqcovbis} are valid with $x=x_0$. We denote 
$$\omega_0:=\frac{d\mu}{d|\mu|}(x_0).$$ 
From Lemma \ref{eqcov} \eqref{iii)} we have the existence of a sequence $r_j$ converging to zero such that  
$$\mu_j\underset{j\rightarrow\infty}{\overset{*}{\rightharpoonup}}\frac{d\mu}{d|\mu|}(x)\,\nu,$$
with $\nu\in Tan(x,|\mu|)$ a positive measure with $\nu(B_\frac 12)>0$. Then we also have 
(cf Theorem 2.12 (3) in \cite{M} applied with $\mu_s\perp g_R\mathcal L^d$, where $g_R(x):=R \,g(Rx)$ and $R>0$)
$$\nu_j:=\frac{(T^{x_0,r_j})_\sharp|\mu|_s}{|\mu|_s (B(x_0,r_j))} \underset{j\rightarrow\infty}{\overset{*}{\rightharpoonup}}\nu,$$
and as $\nu_j\perp\mathcal L^d$, there exists sets $E_j\subseteq B_\frac 12$ where $\nu_j\llcorner B_\frac 12$ concentrates and $\mathcal L^d$ vanishes:
$$\nu_j(E_j)=\nu_j(B_\frac 12),\quad \mathcal L^d(E_j)=0.$$

By using the full information at $(x_0,\omega_0)$ given by Lemma \ref{eqcov} we shall get Proposition~\ref{dissoc} below. Then the two assertions of Proposition~\ref{dissoc} give us the following contradiction:
$$0<\nu(B_\frac 12)=\underset{j\rightarrow\infty}{\lim}\nu_j(E_j)\leq \underset{j\rightarrow\infty}{\lim}|\nu_j-\nu|(E_j)+\nu(E_j)\leq \underset{j\rightarrow\infty}{\lim}|\nu_j-\nu|(B_\frac 12)=0.$$
Therefore to end the proof of Theorem \ref{thWFL1} it remains to show Proposition \ref{dissoc}. 

\begin{prop}\label{dissoc} At least on a subsequence we have 
$$\nu\llcorner B_\frac 12\Lt\mathcal L^d, \quad \underset{j\rightarrow\infty}{\lim}|\nu_j-\nu|(B_\frac 12)=0.$$
\end{prop}
\begin{proof} Let $\chi$ be a smooth cut-off function equal to  $1$ on $B_\frac 12$ and $0$ outside $B_\frac 34$, satisfying $\int \chi (x) dx =1$. 
The conclusion of the Proposition follows if we show that $\chi\nu\in L^1(\mathbb R^d)$ and on a subsequence
\begin{equation}\label{conv}
 \chi \nu_j\overset{L^1(\mathbb R^d)}{\longrightarrow}\chi\nu.
\end{equation} 
As $\nu_j\underset{j\rightarrow\infty}{\overset{*}{\rightharpoonup}}\nu$, it is thus enough to prove that $\{\chi \nu_j\}$ is precompact in $L^1_{loc}(\mathbb R^d)$.

To get the $L^1_{loc}(\mathbb R^d)$-precompactness of $\{\chi \nu_j\}$  we shall consider the operator $B$ from Proposition \ref{eqcovbis}, and use its ellipticity property through an appropriate inversion argument. \medskip

\noindent
{\em{\bf{Step 1: The inversion formula.}}}\par\medskip
As $x_0\in\tilde E$, the properties in Proposition \ref{eqcovbis} are satisfied for $x_0$ and in particular we obtain the existence of $ B\in M_{1\times m}(\Psi^0_{\text{cl}}(\mathbb{R}^d)),Q\in \Psi^0_{\text{cl}}(\mathbb{R}^d), f \in L^1(\mathbb{R}^d,\mathbb R)$ satisfying the identity
$$ B\mu=Qf,$$
the elliptic condition at $(x_0,\omega_0)$:
\begin{equation}\label{ellbis}
b_{0}(x_0,\xi)\omega_0\neq 0, \quad \forall\xi\in\mathbb S^{d-1},
\end{equation}
and moreover $f$ satisfies the properties (iv) in Proposition \ref{eqcovbis}. For simplicity we have considered only one term in the right-hand side instead of a finite sum, as finitely many such terms can be treated similarly. According to Proposition~\ref{basic}, we can pass from left quantization to right quantization, i.e. we shall work with symbols that depend on $y$ and $\xi$ and are independent of the variable $x$, and the ellipticity property \eqref{ellbis} remains valid for them at $y=x_0$. For simplicity we call these symbols again $b_l$ and $q$.
We apply the zoom $\frac{(T^{x_0,r_j})_\sharp\cdot}{|\mu|_s (B_{r_j}(x_0))}$ and get
$$\tilde{B}^j\mu_j=\tilde{Q}^jf_j,$$
where
\begin{equation}
\begin{gathered}
 f_j\mathcal L^d:=\frac{(T^{x_0,r_j})_\sharp f\mathcal L^d}{|\mu|_s  (B(x_0,r_j))},\quad f_j(x)= \frac{ (r_j)^d}{|\mu|_s  (B(x_0,r_j))} f( x_0 + r_j x ),\\
 \tilde b^j(y,\xi):=b(x_0+r_j y,\frac{\xi} {r_j} ),\quad \tilde q^j(y,\xi):=q(x_0+r_j y,\frac{ \xi} { r_j}).
 \end{gathered}
 \end{equation}
 We first microlocalize in frequency to suppress the low frequencies by applying to the left $ (1- \chi(D_x))$, and get 
 \begin{equation}\label{bj}
B^j  \mu_j :=   (1- \chi(D_x))\tilde{B}^j\mu_j= (1- \chi(D_x))\tilde{Q}^jf_j=:Q^j ( f_j) 
 \end{equation}
  with symbols (remark that the following exact formula follows from the independence of $\tilde{b}, \tilde{q} $ with respect to the $x$ variable and~\eqref{quantdte}),
 $$ b^j (y, \xi) = (1- \chi ( \xi)) b(x_0 + r_j y ,\frac{\xi} {r_j} ),\quad q^j(y,\xi):=(1- \chi (\xi))q(x_0+r_j y,\frac{ \xi} {r_j}).$$
 Notice that the symbols $b^j$ and $q^j$ are in $S^0_{\text{cl}} ( \mathbb{R}^d)$ with semi-norms in $S^0_{\text{cl}} ( \mathbb{R}^d)$ uniformly bounded in $j$. Indeed, from~\eqref{bj} we get 
 \begin{multline}\label{unif}
 | \partial_{y}^\alpha \partial_\xi^\beta b^j (y, \xi) |=  \Bigl | \sum_{\delta + \gamma =\beta } \frac{ \beta !}{\delta ! \gamma !} r_j^{|\alpha| -|\gamma|}\partial _{\xi} ^\delta (1- \chi(\xi))  \partial_{y}^{\alpha} \partial_\xi ^\gamma b(x_0+r_j y, \frac \xi {r_j})\Bigr| \\
  \leq Cr_j^{-|\beta|} (1+ \frac{|\xi|} {r_j} )^{- |\beta|}  1_{\frac 1 2 \leq |\xi|}  +C \sum_{\genfrac{}{}{0pt}{1}{\delta \neq 0} {\delta + \gamma = \beta} } r_j^{-|\gamma |} (1+ \frac{|\xi|} {r_j} )^{- |\gamma |}  1_{\frac 1 2 \leq |\xi| \leq 2}  \leq C (1+ |\xi|)^{- |\beta|} 
  \end{multline}
Now we localize in space 
   $$B^j  \chi \mu_j=Q^j \chi f_j +[\chi, Q^j] f_j-[\chi, B^j] \mu_j.$$
Next we shall apply a smoothing operator $\zeta_\epsilon(D_x)$ where  $\zeta_\epsilon (\xi) =  \widehat{\chi} (\epsilon \xi) $. Since
$$ \zeta_\epsilon (D_x) \chi \nu_j=\widehat{\chi} (\epsilon D_x) (\chi \nu_j) = \chi_{\epsilon} \star (\chi \nu_j ) \overset{\epsilon \rightarrow 0}{\rightharpoonup} \chi \nu, \qquad  \chi_{\epsilon} (x) :=\frac {1} {\epsilon^d} \chi ( \frac{ x} {\epsilon}),$$
we can choose a sequence $\epsilon _j<\frac 18 $ converging to $0$ fast enough so that 
$$ u_j :=\zeta_{\epsilon_j}(D_x)\chi \nu_j=  \chi_{\epsilon_j } \star (\chi \nu_j ) \underset{j\rightarrow\infty}{\overset{*}{\rightharpoonup}} \chi \nu.
$$
Thus to get $\{\chi \nu_j\}$ precompact in $L^1_{loc}(\mathbb R^d)$ it is enough to get $\{u_j\}$ precompact in $L^1_{loc}(\mathbb R^d)$.
We have by applying $\zeta_{\epsilon_j}(D_x)$:
\begin{equation}\label{id}
B^j  \omega_0 u_j =B^j    V_j + {Q}^j  g_j \end{equation}
$$- [ \zeta_{\epsilon_j}(D_x)  , B^j ] \chi \mu_j +   [ \zeta_{\epsilon_j} (D_x) , Q^j ] \chi f_j+ \zeta_{\epsilon_j}(D_x)  ([\chi, Q^j] f_j-[\chi, B^j] \mu_j),$$
where 
$$V_j:= \zeta_{\epsilon_j} (D_x) (\omega_0 \chi \nu_j-\chi \mu_j), \qquad g_j = \zeta_{\epsilon_j} (D_x)  \chi f_j $$

Next we consider the family of pseudodifferential operators 
$$B^j_{\omega_0}:=B^j \omega_0,$$ 
of symbols
$$b^j_{\omega_0}(y,\xi):=b^j(y,\xi)\omega_0= (1- \chi (\xi)) b(x_0 + r_j y,\frac{\xi} {r_j})\omega_0,$$
which are bounded in $S_{cl}^0(\mathbb R^d)$ uniformly with respect to $j$. We split:
$$b^j_{\omega_0}(y,\xi)= (1- \chi (\xi)) b_{0} (x_0 + r_j y,{\xi} )\omega_0 +  (1- \chi (\xi)) (b-b_{0}) (x_0 + r_j y,\frac{\xi} {r_j})\omega_0.
$$
From \eqref{ellbis} we get that the principal symbols $b_{0,l}$ satisfy
$$|b_{0} (y, \xi) \omega_0| \geq c >0,$$
for all $\xi\in \mathbb{S}^{d-1}$ and $y$ close to $x_0$, and following the same lines as in~\eqref{unif}, the symbol 
$$(1- \chi (\xi)) (b-b_{0}) (x_0 + r_j y,\frac{\xi} {r_j})\omega_0$$
is uniformly in $S^{-1}$ and bounded by $Cr_j (1+ |\xi|)^{-1}$. 
So for $j$ large enough, $B^j_{\omega_0}$ is elliptic on  $B_1$ uniformly with respect to $j$, in the sense:
$$
\exists\, J, c>0, \quad | b^j_{\omega_0}(y,\xi)\omega_0 |\geq c, \quad  \forall j \geq J, \forall |\xi|\geq 1, \forall y\in B_1.
$$
and we can  approximately invert it locally as follows. Let us consider $\tilde \chi\in C^\infty_0 (B_1)$ equal to $1$ on $B_\frac 78$ thus in particular near the support of $\chi$. For $j\geq J$ we define 
$$ p^j(x, \xi) = \frac{ \chi(x) (1- \chi({\xi} )) } { b^j_{\omega_0} (x, \xi)},$$ 
which is in $S^{0}_{\text{cl}}$  uniformly with respect to $j$,
so that, by symbolic calculus of pseudodifferential operators, and by adding the additional cutt-off $\tilde \chi$ for later use, 
$$ P^j \tilde \chi B^j_{\omega_0} \tilde \chi= \chi (x) (1- \chi({D_x} )) + \widetilde{R}^j= \chi(x) + R^j, 
$$ where   
 $\tilde \chi R^j \tilde \chi = R^j$ and the family of operators $R^j$ 
is a family of pseudodifferential operators uniformly bounded in $\Psi_{cl}^{-1}$, i.e. 
\begin{equation}\label{norms}
\forall \alpha, \beta, \sup _j \sup_{x,\xi}(1+|\xi|)^{1+|\beta |}|\partial_x^\alpha\partial_\xi^\beta (r^j(x,\xi))| <+\infty.
\end{equation}

As $u_j = \chi_{\epsilon_j} \star ( \chi \nu_j)  $ is supported in an $\epsilon_j<\frac 18$ neighborhood of $B_\frac 34$ we have $u_j = \tilde \chi u_j$. Then we can apply $\tilde \chi P^j\tilde \chi$ to \eqref{id} to get:
 \begin{equation}\label{terms}
u_j=-  R^j ( u_j)+\tilde \chi P^j \tilde \chi(B^j  V_j+Q^jg_j)    \end{equation}
$$ - \tilde \chi P^j \tilde \chi [ \zeta_{\epsilon_j}(D_x)  , B^j ] \chi \mu_j +  \tilde \chi P^j \tilde \chi [\zeta_{\epsilon_j} (D_x) , Q^j ]  \chi f_j+ \tilde \chi P^j \tilde \chi\zeta_{\epsilon_j}(D_x) ( [\chi, Q^j] f_j- [\chi, B^j] \mu_j).$$\smallskip

\noindent
{\em{\bf{Step 2: The compactness arguments.}}}\par\medskip
To get  Proposition~\ref{dissoc}  we want to prove that $\{u_j\}$ is precompact in $L^1_{loc}(\mathbb R^d)$, and we have to study the contributions of each terms.\

We start with  the term $R^j ( u_j)=\tilde \chi R^j \tilde \chi u_j$ in \eqref{terms}. From~\eqref{norms} and Corollary~\ref{coro} applied with $\delta=-1$  the operators $ R^j$ are bounded  from  $\mathcal{M}_0$ to $W^{1- \eta, 1}( \mathbb{R}^d)$ for $0<\eta<1$, uniformly in $j$. Also, by the weak convergence of $u_j$ we get that $\tilde\chi u_j$ is uniformly bounded in $\mathcal{M}_0$. Then  Proposition~\ref{compactness} with $s=1-\eta$  implies  that $\{R^j(u_j)\}$ is pre-compact in $L^1_{loc}(\mathbb R^d)$.

The same argument allows to handle the contribution of $ \tilde \chi P^j \tilde \chi [\zeta_{\epsilon_j} (D_x) , Q^j ]  \chi f_j$, by using this time that $\chi f_j$ is uniformly bounded in $\mathcal{M}_0$ due to the first assertion in Proposition \ref{eqcovbis} (iv), and also the fact that  the symbol $\zeta ( \epsilon_j \xi)$ of  the operator $ \zeta_{\epsilon_j} (D_x)$ is in $ S^0_{\text{cl}}( \mathbb{R}^d)$ uniformly with respect to $j$. The sequence $\{  \tilde \chi P^j \tilde \chi [ \zeta_{\epsilon_j}(D_x)  , B^j ]  \chi \mu_j \}$ is also precompact in $L^1$ by the same argument.

Let us now study  the sequence $\{ \tilde \chi P^j \tilde \chi\zeta_{\epsilon_j}(D_x)  [\chi, B^j]\mu_j\}.$
Here the main difference with respect to the previous analysis is that $\mu_j$ is not necessarily bounded in $\mathcal{M}_0$. However, according to Lemma~\ref{eqcov}  ~\eqref{ii)}, $\mu_j$ is bounded in $\mathcal{M}_{d+ 1/2} ( \mathbb{R}^d)$ only, i.e. its mass on balls of radius $R> 1$ can grow at most like $R^{d+ \frac 12}$. From Proposition~\ref{compactbis}, we get  that the family of operators
 $$ \tilde \chi P^j \tilde \chi  \zeta_{\epsilon_j} (D_x) [\chi, B^j] \langle x \rangle ^{d+\frac 12} $$
 is uniformly in $j$ bounded from $\mathcal{M}_0$ to $W^{1- \epsilon, 1}$ for an $\epsilon\in (0,1)$. 
Then the  boundedness of $ \langle x \rangle ^{-(d+\frac 12)}\mu_j $ in $\mathcal{M}_0$ from Lemma~\ref{eqcov}~\eqref{ii)} and Proposition~\ref{compactness} give that $\{ \tilde \chi P^j \tilde \chi\zeta_{\epsilon_j}(D_x)  [\chi, B^j]\mu_j\}$ is precompact in $L^1_{loc}(\mathbb R^d)$.

We get similarly the relative compactness of the sequence 
$ \{\tilde \chi P^j \tilde \chi\zeta_{\epsilon_j}(D_x) [\chi, Q^j] f_j\}$ by using the boundedness of $ \langle x \rangle ^{-(d+\frac 12)}f_j $ in $\mathcal{M}_0$ which follows from the second assertion of Proposition~\ref{eqcovbis}~\eqref{iv)}.

It remains to study the second  term in \eqref{terms}:
$$\tilde \chi P^j \tilde\chi(B^j  V_j+Q^jg_j).$$   
For this we first note that we have the following property:
\begin{equation}\label{lem4.4}\lim_{j\rightarrow + \infty}  \| V_j \|_{L^1} + \| g_j \|_{L^1}=0.\end{equation}
The $L^1-$convergence of $V_j$'s follows from Lemma~\ref{eqcov} (i)-(ii) exactly as in \cite[(2.8)]{DePRi}. For sake of completeness let us recall the short argument.
For $\epsilon_j < \frac 1 8$, the function $V_j$ is supported in $B_{3/4}+ B_{1/8} \subseteq B_{7/8}$, and
$$
\int_{B_{7/8} }|V_j| (x) dx  \leq \int_{B_1} | \omega_0 \chi \nu_j - \chi \mu_j|  \leq \frac{ \bigl |\omega_0 T^{x_0, r_j}_\sharp |\mu|_s - T^{x_0, r_j}_\sharp \mu \bigr| (B_1)}{ |\mu|_s (B(x_0, r_j)}$$
$$ \leq \frac{ \bigl |\omega_0 |\mu|_s - \mu_s \bigr| B(x_0, r_j)}{ |\mu|_s (B(x_0, r_j))} +\frac{ |\mu_a| B (x_0, r_j)}{ |\mu|_s (B(x_0, r_j)}  \leq  \avint_{B(x_0, r_j)} \Bigl| \frac {d\mu} { d|\mu|} (x_0) -\frac {d\mu} { d|\mu|} (x)\Bigr| +\frac{ |\mu_a| B (x_0, r_j)}{ |\mu|_s (B(x_0, r_j)}.$$
Then we deduce from Lemma~\ref{eqcov} (i) the $L^1-$convergence of $V_j$ to zero. On the other hand, from  Proposition~\ref{eqcovbis} (iv), we get  that $ \chi f_j$, and hence also $g_j= \chi_j \star \chi f_j$, converges to $0$ in $L^1$. 

Now we can use $L^1-L^{1,\infty}$ estimates for the $0-$order operators  $ \tilde \chi P^j \tilde \chi B^j_l$ and $ \tilde \chi P^j \tilde \chi Q^j$ that are Calderon-Zygmund operators, the fact that the bound  depends only on a finite number of semi-norms \cite{CoMe}, that these semi-norms are uniformly bounded in $j$, and the convergence \eqref{lem4.4} to get a convergence in measure:
\begin{multline}\label{measure}
\sup_{\lambda\geq 0}\lambda \mathcal L^d(\{|  \tilde \chi P^j \tilde \chi(B^jV_j+Q^jg_j)|>\lambda\})\\
\leq  \sup_{\lambda\geq 0}\lambda\mathcal L^d(\{ \tilde \chi P^j \tilde \chi V_j|>\frac \lambda {2}\})+\sup_{\lambda\geq 0}\lambda\mathcal L^d(\{| \tilde \chi P^j \tilde \chi Q^j g_j|>\frac \lambda {2}\})\\
\leq C \| V_j \|_{L^1} +C \| g_j \|_{L^1}\overset{j\rightarrow\infty}{\longrightarrow}0.
\end{multline}
To conclude, we are going to use following result from \cite{DePRi}.
\begin{lemma}[\protect{\cite[Lemma 2.2]{DePRi}}]\label{passing}
Consider $\{h_j\}$ a sequence of $L^1-$functions supported in $B_1$ satisfying: 
\begin{enumerate}[label=\alph*)]
\item The sequence $h_j$ converges weakly to $0$, $h_j \overset{*}{\rightharpoonup} 0$ in $\mathcal{D}'( \mathbb{R}^d)$
\item The negative part of $h_j$ tends to $0$ in measure
$$ \forall \lambda >0,  \lim_{j\rightarrow + \infty} \mathcal{L}^d ( \{ h_j^-  > \lambda\}) = 0$$
\item The sequence of negative parts $h_j ^-$ is equi-integrable,
$$ \lim_{\mathcal{L}^d(E) \rightarrow 0} \sup_{j \in \mathbb{N}} \int_{E} h_j ^- dx =0
$$
\end{enumerate}
Then 
$$h_j\overset{L^1_{loc}(B_1)}{\longrightarrow}0.
$$
\end{lemma}
\begin{proof}
Let us recall for completeness the short proof from~\cite{DePRi}. 
Let $\varphi \in C^\infty_0 (B_1)$ with $0 \leq \varphi \leq 1$. Since $h_j$ is supported in $B_1$, it is enough to prove 
$$ \lim_{j\rightarrow +\infty} \int_{B_1} \varphi |h_j| dx =0.$$
We write 
$$ \int_{B_1} \varphi |h_j| dx= \int_{B_1} \varphi h_j dx + 2 \int_{B_1} \varphi h_j^- dx \leq \int_{B_1} \varphi h_j dx + 2 \int_{B_1}  h_j^- dx.$$
As $h_j \overset{*}{\rightharpoonup} 0$ it is enough to show that the last integral converges to $0$. 
Then from the equi-integrability, for any $\varepsilon >0$ there exists $\delta >0$ such that 
$$ \mathcal{L}^d(E) \leq \delta \Rightarrow  \sup_{j \in \mathbb{N}} \int_{E} h_j ^- dx \leq \varepsilon.$$
From the convergence in measure we have 
$$ \exists J; \forall j \geq J, \mathcal{L}^d( \{ h_j ^- > \varepsilon\} \leq \delta.$$ 
We deduce that for $j\geq J$
$$ \int_{B_1 } h^-_j dx \leq \int _{ \{ h_j ^- > \varepsilon\}\cap B_1} h^-_j + \int _{ \{ h_j ^- \leq \varepsilon\}\cap B_1} h^-_j \leq \varepsilon (1+ \mathcal{L}^d(B_1)).
$$
\end{proof}
We now check that 
$$ h_j := \tilde \chi P^j \tilde\chi(B^j V_j+Q^jg_j) , $$ 
satisfies the assumptions of Lemma~\ref{passing}: 
\begin{enumerate}[label=\alph*)]
\item from \eqref{lem4.4} we have that $V_j$ and $g_j$ are converging to $0$ in $L^1$, thus $h_j$ converges weakly to $0$,
\item from~\eqref{measure}, $h_j$ and hence also $h_j^-$ tends to $0$ in measure,
\item as $u_j\geq 0$, the negative part of $ h_j$ is bounded by $ (\chi \nu_j- h_j)^+$ which we already proved it  is relatively compact in $L^1_{loc}(\mathbb R^d)$.  As a consequence, the negative part of $h_j$ is equi-integrable.
\end{enumerate}
Thus by applying Lemma~\ref{passing} we deduce that the sequence $\{h_j\}$ converges to $0$ in $L^1_{loc}(B_1)$. 
This concludes the fact that $u_j$ converges in $L^1_{loc}(\mathbb R^d)$, and finishes the proof of Proposition~\ref{dissoc} and hence of Theorem~\ref{thWFL1}.
\end{proof}

\section{Proof of Theorems \ref{corWFL1},  \ref{corellL2}, and \ref{corellL1}} \label{sectcorWFl1}

\begin{proof}[Proof of Theorem~\ref{corWFL1}]
The first assertion of Theorems \ref{corWFL1} is equivalent to 
$$ \mu \in \mathcal{M}_t, \Pi_{13}(WF_{L^1}(\mu)))=\varnothing\iff \mu\in L^1_{loc}.$$
If $\mu\in L^1_{loc}$ then all points $(x,\omega)\in\mathbb{R}^d\times \mathbb R^m$ are in the complementary of the projection of the wave front set, $\Pi_{13}(WF_{L^1}(\mu)))$ since \eqref{WFL1} is satisfied with $B=\chi\in C^\infty_c( \mathbb{R}^d)$ equal to $1$ near $x$.
Conversely, if $\Pi_{13}(WF_{L^1}(\mu)))=\varnothing$ then Theorem \ref{thWFL1} ensures  that $\mu\in L^1$.

Concerning the second assertion, let us first consider $x\notin \mbox{supp}\, |\mu|_s$. Then there exists $\chi\in C^\infty_c(\mathbb{R}^d)$ equal to $1$ near $x$ and such that $\chi\mu\in L^1$. The first assertion ensures that $WF_{L^1}(\chi \mu)=\varnothing$, so using \eqref{WFloc} we get $x\notin \Pi_{1}(WF_{L^1}(\mu)))$. So we have the inclusion $\Pi_{1}(WF_{L^1}(\mu)))\subseteq \mbox{supp}\, |\mu|_s.$
On the other hand Theorem \ref{thWFL1} gives us the existence of a set $N$ of null $|\mu|_s$-measure such that for $x\in N^c$ we have $(x, \frac{d\mu}{d|\mu|}(x))\in \Pi_{13}(WF_{L^1}(\mu)))$ so in particular $x\in \Pi_{1}(WF_{L^1}(\mu)))$. Therefore we get the last assertion. 

\end{proof}

\begin{proof}[Proof of Theorems~\ref{corellL2} and \ref{corellL1}]  Theorem~\ref{corellL2} is clearly a consequence of Theorem~\ref{corellL1}. By the same argument as in the proof of the last point in Remark~\ref{front-onde-rem},  we can assume that $\mu$ is supported in a neighborhood of $x_0$ where the elliptic assumption holds. Now, by standard elliptic regularity we can invert $A$ on the support of $\mu$ and get that 
$$ \mu \in W^{k-\epsilon, 1},\, \forall \epsilon >0,$$ 
the $\epsilon$-loss coming from the defect of $L^1$-boundedness of $0$-th order operators.
  We now come back to  
$$ A (B\mu) =  (AB - CA) \mu + CA \mu \in  \Psi^k(W^{1- 2\epsilon, 1}) +\Psi^k (L^1) \subset \Psi^k ( L^1),
$$ 
(by choosing $2\epsilon <1$)
because according to~\eqref{commutateur}, $AB- CA  \in \Psi^{2k-1}$. 
We deduce from the third point in Remark~\ref{front-onde-rem} that $WF_{L^1} (B\mu) = \varnothing$.
From Theorem~\ref{corWFL1} we get Theorem~\ref{corellL1}.

\end{proof}

\section{On the singular part of elementary constrained measures and more}\label{sectpropag}
In \cite{DePRi} was considered the question of elementary constrained measures, i.e. measures of the form $\mu_0=\omega_0  \nu$ with $\omega_0\in\mathbb R^{m*}, \nu\in\mathcal M_t(\mathbb R^d,\mathbb R)$, constrained to vanish under the action of a first order linear constant coefficient operator $\sum_{|\alpha|= 1}A_\alpha\partial_x^\alpha $, $A_\alpha\in\mathbb R^{n\times m}$, satisfying
$$C:=\{\xi\in\mathbb R^d, \sum_{|\alpha|=1}A_\alpha \omega_0\,\xi^\alpha=0 \}\neq \{0_{\mathbb R^d}\}.$$
It was noticed in \cite{DePRi} that in this case, passing in Fourier the equation yields $\mbox{supp }\hat\nu\subseteq C$ and $\nu$ is invariant in the directions orthogonal to $C$. Indeed, for $\tilde x\in C^\perp$, using that $\widehat{\nu}$ is supported in $C$,
 $$\nu(x+\tilde x)=\int e^{i(x+\tilde x)\xi}d \hat\nu(\xi)=\int_{C}  e^{i(x+\tilde x)\xi}d \hat\nu(\xi)=\nu(x).$$
 However, this property gives useful informations only for first order operators as in this case, the set $C$ is a vector space. For higher order operators the example of a simple scalar wave equation ($n=m=1$, $\partial_t^2 - \Delta$) for which the characteristic manifold is 
 $$C= \{ (\tau, \eta)\in \mathbb{R}^{d+1}; \tau ^2 = |\eta|^2\},$$
whose orthogonal set $C^\perp$ is reduced to $\{0_{\mathbb R^{d+1}}\}$, shows that this invariance property may provide no information. In this section, we give a few elements toward the understanding of more general cases, providing information about the structure of polarisation of the singular part of a constrained measure in non-elementary constrained measure cases. We shall not use our wave front $WF_{L^1}$, but rely rather on propagation of singularities ideas introduced previously for the study of systems of PDE's, and in particular systems of wave equations (see~\cite{BuLe}). The results in this section can be seen as a very first step toward of a general theory of microlocal defect measures at the $L^1$ level, which is still quite far away!
  
We first notice that  the special choice $\mu_0 = \omega_0  \nu$ reduces the study to a system of $n$  scalar equations on the measure~$\nu$,
$$ A_j \nu = \sum_{|\alpha |=1} A^j_\alpha \omega_0 \partial_x^\alpha \nu=0, $$
where $A^j_\alpha$ is the $j$-th line of the matrix $A_\alpha$, and the invariance along $C^\perp$ is just the invariance of $\nu$ by each of the $n$ transport equations $A_j$. 
When studying propagation of singularities for systems, the natural extension of scalar equations (see~\cite[Sections 3 \& 4]{De} and~\cite{BuLe} for boundary value problems) is to study systems with diagonal (or at least diagonalisable) {\em scalar} principal parts, and we start with an elementary result in this simpler case.
\begin{lemma}\label{inv1}
Consider a smooth vector field  on $\mathbb{R}^d$
$$A:= \sum_{i=1}^d b_i (x) \,\partial_{x_i}, \quad ^t A:= A + \divbis (b) =  \sum_{i=1}^d\partial_{x_i}\circ b_i,$$ 
a function $ H\in C( \mathbb{R}^d, M_{m\times m} (\mathbb{R}))$
 and  $\mu_0 \in \mathcal{M}_{\text{loc}} ( \mathbb{R}^d, \mathbb{R}^m)$ solution of the system\footnote{Abusing notations we still denote by $A$ the vector field $A \text{Id}_{\mathbb{R}^m}$ and $^tA= \sum_i \partial_{x_i} \circ b_i\,  \text{Id}_{\mathbb{R}^m}$ its transpose.}
\begin{equation}\label{transport1}
 ^t A \mu_0 + H \mu_0 =\divbis (b \mu_0) + H \mu = \Bigl( \sum_{i=1}^d \partial_{x_j} (b_i \,\mu_{0,k}) + \sum_{p=1}^m h_{k,p} \, \mu_{0,p}\Bigr)_{k=1, \dots, m}=0_{\mathbb R^m}.
 \end{equation}
Then the set
$$ Z= \{ \bigl(x, \frac{d\mu_0}{d|\mu_0|} (x)\bigr), x\in \mathbb{R}^d\} , \qquad (\text{ resp. }Z_s= \{ \bigl(x, \frac{d\mu_{0,s}}{d|\mu_{0,s}|} (x)\bigr), x\in \mathbb{R}^d\} )$$ 
is $|\mu_0|$-a.e. (resp. $|\mu_{0,s}|$-a.e) invariant by the flow 
$$ (x_0, \omega_0) \mapsto  \bigl(x(s), \frac{\omega(s)}{\| \omega(s)\|}\bigr), $$
where $x(s)= \phi(s,x_0)$ and $\omega(s)= \widetilde{\phi} (s,x_0, \omega_0)$ are defined by 
$$ \dot{x} (s) = b(x(s)),\; x(0) =x_0 , \qquad \dot{\omega} (s) = - H (x(s)) \omega(s), \; \omega(0) =\omega_0 .$$
\end{lemma}
\begin{proof} The proof is easy: we just solve the equation! More precisely, we solve the associated equation 
\begin{equation}\label{transport2}
 \partial_s \mu +\divbis (b \mu) + H \mu =0, \mu_{\vert_{s=0}} = \mu_0 
 \end{equation}
Since by the duality method the solution to this equation is unique, if $\mu_0$ solves~\eqref{transport1}, then the unique solution to~\eqref{transport2} is given by $\mu(s) \equiv \mu_0$. Let us now assume that $\mu$ solves~\eqref{transport2} and define $\nu$ by 
 $$ \nu = \phi(-s, \cdot )_\sharp \mu \Leftrightarrow \mu = \phi(s, \cdot )_\sharp \nu.$$
Here we abuse slightly notations and denote by $\phi(s,\cdot)_\sharp$ the push forward of the measure $\mu$ by the map
$$ (s,x) \mapsto (s, \phi(s,x)),$$
defined by 
$$ \langle \phi(s,\cdot)_\sharp \mu, \psi\rangle = \langle \mu, \psi (s, \phi(s,x))\rangle,$$
for any $ \psi \in C^\infty_0 ( \mathbb{R}^{d+1},\mathbb R ^m).$ By using the definition of $\phi$ we get
$$ \langle \partial_s \mu +\divbis (b \mu)+H \mu , \psi \rangle = - \langle \mu , ( \partial_1 + b .\nabla _x ) \psi\rangle  +\langle H \mu , \psi \rangle $$
$$  = -\langle \nu,  ( \partial_1 \psi + b. \nabla _x \psi ) _{|_{(s, \phi(s,x))}} \rangle +\langle H \mu , \psi \rangle 
  = -\langle \nu,   \partial_s  (\psi (s, \phi(s,x))) \rangle +\langle H \mu , \psi \rangle $$
  $$= \langle  \partial_s \nu,  \psi _{|_{(s, \phi(s,x))}} \rangle+\langle H(\cdot)\, \phi(s, \cdot )_\sharp \nu , \psi \rangle 
  =  \langle  \phi(s, \cdot )_\sharp (\partial_s \nu + H(\phi(s,x))\nu ),  \psi \rangle.
$$
 We deduce that $\mu$ solves~\eqref{transport2} if and only if $\nu= \phi(-s, \cdot) _\sharp \mu(s, \cdot)$ solves 
$$
\partial_s \nu + H(\phi) \nu=0, \quad \nu_{\vert_{s=0}} = \mu_0.
$$
 To solve this equation we apply the variation of parameter method and compute, with $C(s,x)$ to be defined,  
$$
  \partial_s ( C(s,x)  \,\nu)(s,x) =\left(\partial_s C(s,x)  - C(s,x) H(\phi(s,x) \right) \nu (s,x).
$$
 Let us now define $C$ as the solution to the differential equation
 $$ \partial_s C (s,x) =  C(s,x) H(\phi(s,x)), \quad C_{\vert_{s=0}} = \text{Id}.$$ 
 Remark that $C$ is invertible as $C^{-1}$ is the solution to  
 $$ \partial_s C^{-1} (s,x)= -H( \phi(s,x))C^{-1} (s,x) , \quad {C^{-1}}_{\vert_{s=0}}= \text{ Id} .$$
 We get then
 $$ C(s,x) \nu (t, \cdot) = C(0,x) \nu(0, \cdot) = \mu_0 .$$
 Summarising, we proved 
 $$\mu_0 = \mu(s,\cdot ) =  \phi(s,\cdot) _\sharp \nu (s, \cdot) = \phi(s,\cdot) _\sharp \left( C^{-1} (s,x) \mu_0\right) \Leftrightarrow   \phi(-s,\cdot) _\sharp \mu_0 = C^{-1} (s,x) \mu_0.$$
This implies that the polarisation $\frac{d\mu_0}{d|\mu_0|}$ of $\mu_0$ at $\phi(s,x)$ is  colinear to $C^{-1}(s,x) $ applied to the polarisation of $\mu_0$ at $x$ (colinear rather than equal because the Jacobian determinant of the change of variables $x \mapsto \phi(s, x)$ is not necessarilly equal to $1$ and $C^{-1}$ is not necessarily an isometry).    
\end{proof}
Now, we turn to a more complicated setting involving a first order {\em non diagonal} equation. We consider two smooth vector fields $A_1, A_2$ given as previously by functions $b_1$ and $b_2$ on $\mathbb{R}^2$ and $\mu = (\mu_1, \mu_2) \in \mathcal{M}_{loc}( \mathbb{R}^2, \mathbb{R}^{m_1+ m_2})$ solution to the coupled system
\begin{equation}\label{systmu}
 \begin{pmatrix} ^tA_1& 0 \\ 0 & ^t A_2 \end{pmatrix} \begin{pmatrix} \mu_1 \\ \mu_2\end{pmatrix} + \begin{pmatrix}H_{11} & H_{12} \\ H_{21} & H_{22} \end{pmatrix} \begin{pmatrix} \mu_1 \\ \mu_2\end{pmatrix} =0
\end{equation}
where we now assume for simplicity that the matrices $H_{jk}$ is $C^\infty$. We can now state our result which is reminiscent of propagation of singularities type results. 
\begin{theorem}\label{uncouple}
Assume that the vector fields $A_1$ and $A_2$ are linearly independent at each point $x\in \mathbb{R}^2$, and that $\mu$ solves \eqref{systmu}. Then for $j\in\{1,2\}$ the sets 
$$ Z_j= \{ \bigl(x, \frac{d\mu_j}{d|\mu_j|} (x)\bigr), x\in \mathbb{R}^d\} $$ 
are $|\mu_{j,s}|$-a.e. invariant by the flows 
$$ (x_0, \omega_0) \mapsto (x_j(s), \frac{\omega_j(s)}{\|\omega_j(s)\|}), $$
where 
$$ \dot{x}_j (s) = b_j(s), \; x_j(0) =x_0 , \quad \dot{\omega}_j (s) = - H_{jj} (x_j(s)) \,\omega_j(s),\;\omega_j(0) =\omega_0 .$$
In other words, as far as their singular parts are concerned, the propagation formulas for $\mu_{j,s}$ are obtained by forgetting the coupling terms $H_{12}$ and $H_{21}$ in the equation.
\end{theorem}
\begin{remark}
For conciseness, we chose to work with smooth vector fields $A_j$. It is however very likely that this kind of results holds for Lipschitz vector fields. It would be interesting to apply the methods developed in~\cite{AmCr} to deal with lower regularity.
\end{remark}
\begin{proof} 
We reduce the proof to Lemma~\ref{inv1} by showing that locally we have the uncoupling property:
\begin{equation}\label{uncoupled}
^tA_j\, \mu_{j,s} + H_{jj} \,\mu_{j,s}=0.
\end{equation}
We work near a point $x_0 \in \mathbb{R}^2$ and can replace $\mu_j$ by $\chi(x) \mu_j$, $\chi\in C^\infty_0 ( \mathbb{R}^2)$ equal to $1$ near $x_0$ and after a linear change of variables, we can assume that 
$$ A_1 (x_0 ) =\, ^tA_1 (x_0 ) = \frac{ \partial}{\partial x_1}, \quad A_2 (x_0 ) =\, ^tA_2 (x_0 ) = \frac{ \partial}{\partial x_2}.$$
Consider now  a smooth cut off $\zeta_1(\xi) $ vanishing near $0$, homogeneous of degree $0$ outside $\{ \| \xi\| \geq 1\}$ and  equal to $1$ in a small conical neighborhood of $(0,1)\cup (0, -1)$. The first step is the following Lemma. 
\begin{lemma}\label{regularite}
Near $x_0$ we have
$$ (1- \zeta_1)(D_x)  \mu_1 \in W^{1-\epsilon, 1}( \mathbb{R}^2).
$$ 
Notice also that this implies $\mu_{1,s}= (\zeta_1(D_x)  \mu_1)_s$.
\end{lemma}
\begin{proof}Let $\zeta_0 \in C^\infty_0(\mathbb{R}^2)$ equal to $1$ near $0$. Since $\zeta_0 (D_x) \in \Psi^{-N}$, we have from Corollary \ref{coro}:
$$ (1- \zeta_1)\zeta_0 (D_x)  \mu_1 \in W^{1-\epsilon, 1}( \mathbb{R}^2).
$$ 
It remains to study 
$$  (1- \zeta_1)(1- \zeta_0) (D_x) \mu_1.
$$
Near $x_0$ the principal symbol $a_1(x, \xi)$ of the operator $^tA_1$  is close to its value at $(x_0, \xi)$ which is $i \xi_1$, hence it is invertible in a neighborhood of  the support of $(1- \zeta_1)(1- \zeta_0)(\xi)$. As a consequence, we can define 
$$ p(x,\xi)= \frac{\chi(x) (1- \zeta_1)(1- \zeta_0)( \xi) } {a_1(x, \xi)}, $$ 
with $\chi\in C^\infty_0$ smooth, equal to $1$ near $x_0$. Applying $p(x, D_x) $ to the equation satisfied by $\mu_1$,
\begin{equation}\label{eqmu1}
 ^tA_1 \, \mu_1 = - H_{11} \,\mu_1 - H_{12}\, \mu_2,
 \end{equation}
we  get by symbolic calculus 
$$ \chi(x)  (1- \zeta_1)(1- \zeta_0)( D_x) \mu_1 = R_1\mu_1 + R_2 \mu_2,$$
where $R_j$ are matrices of pseudodifferential operators of order $-1$, which implies by Corollary \ref{coro} that
$$ \chi(x)  (1- \zeta_1)(1- \zeta_0)( D_x) \mu_1 \in W^{1- \epsilon, 1} (\mathbb{R}^2).$$
\end{proof}

Let now $\widetilde{\zeta_1}$ be a cut-off  equal to $1$ in a neighborhood of the support of $\zeta_1$. 
 Applying $\zeta_1( D_x)$ to equation \eqref{eqmu1} we get 
\begin{equation}\label{new}
  ^tA_1 \,\zeta_1( D_x) \mu_1 + H_{11} \,\zeta_1( D_x) \mu_1 
  \end{equation}
  $$= - H_{12} \,\zeta_1( D_x) \mu_2 +  [ ^tA_1, \zeta_1( D_x) ] \mu_1 + [H_{11},  \zeta_1( D_x) ] \mu_1 + [ H_{12} , \zeta_1( D_x)] \mu_2 ,
$$
Now the key point is that the r.h.s.  is an $L^1$ function in a neighborhood of $x_0$. Indeed, it is clear for the two last term as they are operators of order $-1$ applied to measures so Corollary \ref{coro} can be used\footnote{This is where we use the smoothness of the $A_{j,k}$; this smoothness could be relaxed to H\"older continuity or even log continuity
$$ \|A(x) - A(y)\| \leq \frac{C} { |\log ^\alpha(\|x-y\|)|} , \alpha >1, \| x-y\| \leq \frac 1 2.$$}.
Let us now study the first term. The function $\zeta_1$ is supported in a small conical neighborhood of $(0,1 ) \cup (0, -1) $. We can now choose a smooth function $\zeta_2$ vanishing near $0$, homogeneous of degree $0$ outside $\{ \| \xi\| \geq 1\}$ and  equal to $1$ in a small conical neighborhood of $(1,0)\cup (-1,0)$, and which vanishes on the support of $\zeta_1$.  This is where we use crucially that we are working in $\mathbb{R}^2$:  in higher dimensions, the cut-off $\zeta_1$ would be required to vanish near the characteristic manifold of $X_1$, which at $x_0$ is 
$$ C_1 = \{ \xi; \xi_1=0 \}$$ 
while the cut-off $\zeta_2 $ is required to vanish on the characteristic manifold of $X_2$ which is at~$x_0$,
$$ C_2 = \{ \xi; \xi_2=0 \}.$$
In dimension $2$, these two manifolds intersect at $0_{\mathbb R^2}$. In higher dimension they intersect along the plane 
$$ C_{1,2}= \{ \xi; \xi_1= \xi_2=0 \},$$ 
and consequently such a choice for $\zeta_2$ (equal to $1$ near $C_1$ but vanishing near $C_2$, apart from a neighborhood of $0$) is possible only in dimension $2$. 
\begin{figure}[h]\label{fig-1}
\begin{center}
\begin{tikzpicture}[scale =1]
\draw [thick, blue] (-0.6, 2.75) -- (-0.3,0.75);
\draw [thick, blue](-0.3,0.75)  arc (-165:-15:0.31);
\draw [thick, blue](0.6,2.75)-- (0.3, 0.75) ;
\draw [thick, blue] (-0.6, -2.75) -- (-0.3,-0.75);
\draw [thick, blue](-0.3,-0.75)  arc (165:15:0.31);
\draw [thick, blue](0.6,-2.75)-- (0.3, -0.75) ;
\draw [thick, blue] ( 2.75, 0.6) -- (0.75, 0.3);
\draw [thick, blue](0.75, 0.3)  arc (-255:-105:0.31);
\draw [thick, blue](2.75, -0.6)-- ( 0.75, -0.3) ;
\draw [thick, blue] ( -2.75, 0.6) -- (-0.75, 0.3);
\draw [thick, blue](-0.75, 0.3)  arc (75: -75:0.31);
\draw [thick, blue](-2.75, -0.6)-- ( -0.75, -0.3) ;
\draw (2,0.3) node[right]{$\zeta_2 =1$};
\draw (-2,0.3) node[left]{$\zeta_2 =1$};
\draw (-0.66,2.5) node[right]{$\zeta_1 =1$};
\draw (-0.66,-2.50) node[right]{$\zeta_1 =1$};
\draw[->] (-3,0) -- (3,0); 
\draw (3,0) node[right] {$\xi_1$}; 
\draw [->] (0,-3) -- (0,3); 
\draw (0,3) node[above] {$\xi_2$};
\end{tikzpicture}
\end{center}
\caption{The cut-off functions}
\end{figure}
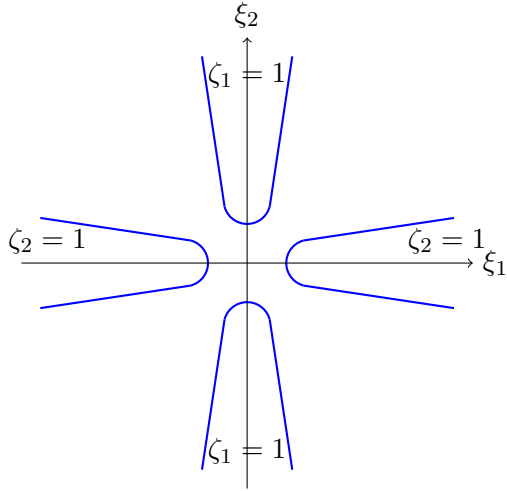
Applying Lemma~\ref{regularite} exchanging the roles of $\mu_1$ and $\mu_2$ (and the roles of the variables $x_1$ and $x_2$), we get by Proposition \ref{compact} that
$$ (1- \zeta_2) (D_x) \mu_2 \in W^{1-\epsilon, 1}( \mathbb{R}^2)\Rightarrow \zeta_1( D_x) \mu_2 = \zeta_1( D_x)  (1- \zeta_2) (D_x)\mu_2 \in W^{1-2\epsilon, 1}( \mathbb{R}^2),$$
and consequently near $x_0$,
$$H_{12}\, \zeta_1( D_x) \mu_2 \in L^1.$$
Finally, to study the second term in the r.h.s. of~\eqref{new}, we apply the symbolic calculus formula \eqref{explicit} which shows that 
$$  [ ^t A_1, \zeta_1( D_x) ] =  -( \nabla_x a_1.\nabla_\xi \zeta_1) (x, D_x) + R, \quad R\in \Psi^{-1}.$$ 
Now $R \mu_1\in W^{1-\epsilon , 1} ( \mathbb{R}^2)$, and since $\zeta_1=1$ on a neighborhood of the support of $\widetilde{\zeta_1}$, we get that $\nabla_\xi \zeta_1$ is supported  where $(1- \widetilde{\zeta_1})=1$ and consequently we have 
$$  (\nabla_x a_1 . \nabla_\xi \zeta_1) (x, D_x) =  (\nabla_x a_1 . \nabla_\xi \zeta_1 (1- \widetilde{\zeta_1}))(x, D_x), $$
therefore Lemma \ref{regularite} with $\tilde \zeta_1$ and Proposition \ref{compact} imply
$$(\nabla_x a_1 . \nabla_\xi \zeta_1) (x, D_x)  \mu_1= ( \nabla_x a_1 . \nabla_\xi \zeta_1) (x, D_x) (1- \widetilde{\zeta_1})(D_x) \mu_1\in W^{1- 2\epsilon,1} ( \mathbb{R}^2).$$
Summarizing, we have obtained from \eqref{new} that in a neighborhood of $x_0$
$$  ^tA_1\,  \zeta_1( D_x) \mu_1 + H_{11}\, \zeta_1( D_x) \mu_1\in L^1( \mathbb{R}^d).$$

We can now revisit the proof of Lemma~\ref{inv1}, with $\mu_0$ replaced by $\zeta_1(D_x) \mu_1$ (wich, according to Lemma~\ref{regularite} is also a measure), with the only modification that we now have an $L^1$ r.h.s. 
We get 
$$ \phi_1(-s,\cdot) ^\sharp \bigl( \zeta_1(D_x)\mu_{1}\bigr) - C_1^{-1} (s,x) \bigl( \zeta_1(D_x)\mu_{1}\bigr)  \in L^1_{\text{loc}},$$ 
Passing to the singular parts (from Lemma~\ref{regularite} $\mu_{1,s}= \zeta_1(D_x)  \mu_1)_s$), we get near $x_0$ (and for small $s$)
$$\phi_1(-s,\cdot) ^\sharp \mu_{1,s}  =  C_1^{-1} (s,x) \mu_{1,s},$$ 
which implies Theorem~\ref{uncouple} (and also~\eqref{uncoupled}) for $\mu_1$  and small $s$. The general case is obtained by iterating in $s$.  The proof for $\mu_2$ is similar.

\end{proof}

%
%
%
%

\section{Pseudodifferential operators}\label{sec.pseudo}
In this section we have gathered basic facts about pseudodifferential operators. We refer to~\cite[Chapter XVIII]{Ho} for a general presentation or~\cite[Section VI.6]{Stein} for a presentation closer to our needs of the pseudodifferential calculus (see also \cite{AlGe}).
Let us recall the definitions of the symbol classes. Here we adopt the following convention:
$$ \forall \alpha \in \mathbb{N}^d, |\alpha | = \sum_{j=1}^d \alpha _j, \quad \alpha ! = \prod_{j=1}^d \alpha_j !,\quad  \partial_x ^\alpha = \frac{\partial^{\alpha_1}} { \partial x_1^ {\alpha_1}} \circ \dots \circ  \frac{\partial^{\alpha_d}} { \partial x_d^ {\alpha^d}}.
$$
We now define the class of symbols of order $k$ by 
\begin{multline}
 S^k( \mathbb{R}^d) =\{a\in C^\infty( \mathbb{R}^{3d});  \\
 \forall \alpha, \beta, \gamma \in \mathbb{N}^d, \sup_{x,y,\xi \in \mathbb{R}^d}  |\partial_x^\alpha \partial_y^\beta \partial_\xi^\gamma a(x, y , \xi)| (1+ |\xi|)^{-k+|\gamma|}=: \|a\|_{k, \alpha, \beta, \gamma}<+\infty\}.
\end{multline}
The constants $\|a\|_{\alpha, \beta, \gamma} $ are called semi-norms of the symbol $a$.
Most of the time, the symbols we shall consider will not depend on the $y$ variable, but it is convenient to allow this dependence.
For simplicity we shall sometimes only consider the subclass $S^k_{\text{cl}}$ of symbols  admiting homogeneous principal symbol i.e. there exists in addition $\chi \in C^\infty_c( \mathbb{R}^d)$ and $\tilde a_k\in C^\infty (\mathbb{R}^d\times \mathbb{S}^d)$ such that:
$$ a - (1-\chi)(\xi) |\xi|^k \tilde a_k (x , \frac \xi{|\xi|}) \in S^{k-1}.$$
The function $a_k(x,\xi):=|\xi|^k\tilde a_k(x , \frac \xi{|\xi|})$ is the principal symbol of $a$. 

To any symbol $a \in S^k_{\text{cl}}(\mathbb{R}^d)$ we can associate an operator on the temperate distributions set $\mathcal{S}'( \mathbb{R}^d)$ by the formula
$$ a(x,y, D_x) u (x)= \text{Op} (a) u(x)= \frac{ 1 }{(2\pi)^d} \int e^{i(x-y) \cdot \xi} a(x, y , \xi) u(y) dy d\xi,$$
where this integral is defined as an oscillatory integral. We shall denote $Op (a) \in \Psi^k_{\text{cl}}( \mathbb{R}^d)$.
Remark that if $a$ do not depend on the variable $x$, we have 
\begin{equation}\label{quantdte}
  \text{Op} (a) u(x) = \mathcal{F}^{-1} \left( \int e^{-iy \cdot \xi} a( y , \xi) u(y) dy d\xi\right)(x).
 \end{equation}
 
An operator $Op(a)$ is said to be elliptic at a point $(x_0, \xi_0)$ if the principal symbol $a_k(x,y, \xi)$ is non zero at the point $(x_0,x_0, \frac{\xi_0} {|\xi_0|})$.  
\begin{definition}
For any $a\in S^{k}( \mathbb{R}^d)$ and any sequence $a_i \in S^{k_i}( \mathbb{R}^d)$ with $k_0 =k > k_1 > \dots $, we write 
$$ a \sim \sum_i a_i \Leftrightarrow \forall M, (a- \sum_{i=0}^N a_i ) \in S^{k_{N+1}}.$$
\end{definition}
The basic properties of symbolic calculus of pseudodifferential operators are summarized in the following 
\begin{prop}\label{basic}
We have the following symbolic calculus properties
\begin{itemize}
\item
For any $a\in S^k_{\text{cl}}(\mathbb{R}^d)$, there exists $\widetilde{a} \in S^k_{\text{cl}}(\mathbb{R}^d)$ {\em not depending on the variable $y$}. (resp. $\widetilde{\widetilde{a}}\in S^k_{\text{cl}}(\mathbb{R}^d)$ {\em not depending on the variable $x$}), such that 
$$Op(a) = Op( \widetilde{a}) , (\text{resp. } Op(a) = Op( \widetilde{\widetilde{a}}),
$$
with 
$$\widetilde{a} (x,\xi) \sim \sum_{N} \sum_{|\alpha |\leq N } \frac{i^{|\alpha|}} { \alpha !} (\partial_y^\alpha \partial_\xi^\alpha a(x,y,\xi)) \mid_{y=x},
$$
$$\widetilde{\widetilde{a}} (y,\xi) \sim \sum_{N} \sum_{|\alpha |\leq N } \frac{(-i)^{|\alpha|}} { \alpha !} (\partial_x^\alpha \partial_\xi^\alpha a(x,y,\xi)) \mid_{x=y}.
$$
In particular, 
$$\tilde a_k(x,\xi)=a(x,x,\xi),\quad \widetilde{\widetilde{a}}_k (y,\xi)=a(y,y,\xi),$$
and $Op(a)$ is elliptic at a point $(x_0, \xi_0)$ if and only if $Op(\tilde a)$ is elliptic at $(x_0, \xi_0)$ in the sense that $\tilde a_k(x_0,\xi_0)\neq 0$ (and similarly for the right quantization).
\item
The formal $L^2$ adjoint of a pseudodifferential operator, $Op(a)$, is the pseudodifferential operator $Op(a^*)$, with 
$$ a^*(x,y,\xi) = \overline{a(y,x,\xi)}.$$
\item
For any $a\in S^k_{\text{cl}}( \mathbb{R}^d), b\in S^{\tilde k}( \mathbb{R}^d)$, there exists $c \in S^{k+\tilde k}( \mathbb{R}^d)$ such that
$$ Op(a) \circ Op(b) = Op(c).$$
Furthermore, if the symbols $a$, $b$ and $c$ depend only on the $x$ variable (by the previous results we can reduce the analysis to this case), we have
\begin{equation}\label{compo}
 c= a\sharp b \sim \sum_{N} \sum_{|\alpha| =N} \frac{i^{|\alpha|}}{ \alpha !} \partial_\xi^\alpha (a(x, \xi)) \partial_x^\alpha (b(x, \xi)).
\end{equation}\end{itemize}
\end{prop}
\begin{rem}
From the explicit formula
$$a\sharp b (x, \xi) = \frac 1 {(2\pi)^d} \int e^{i(x-y)\cdot (\eta-\xi)} a(x, \eta) b(y, \xi) dy d\eta,
$$ 
we can actually get a quantitative version of~\eqref{compo}. Namely, for any $N_0$, each $S^{k+\tilde k-N_0-1}(\mathbb R^d)$ semi-norm of
\begin{equation}\label{explicit}
a\sharp b- \sum_{N\leq N_0}\sum_{|\alpha| =N} \frac{i^{|\alpha|}}{ \alpha !} \partial_\xi^\alpha (a(x, \xi)) \partial_x^\alpha (b(x, \xi)) 
\end{equation}
is bounded by a product of a finite number of semi-norms of $a$ and $b$.
\end{rem}
 Recal that the Sobolev space $W^{s,p} ( \mathbb{R}^d)$ is defined by 
$$ W^{s,p} ( \mathbb{R}^d)= \{ u \in \mathcal{S}'( \mathbb{R}^d); (1- \Delta )^{s/2} u \in L^p( \mathbb{R}^d)\} ,$$
and that we have the following property.
\begin{prop}\label{compactness}
For any $1\leq p \leq +\infty$ and any $s>0$, $\chi\in C^\infty_0 ( \mathbb{R}^d)$, the application
$$ u \in W^{s,p} ( \mathbb{R}^d) \mapsto \chi u \in L^p (\mathbb{R}^d)
$$ is compact. 
\end{prop}
Pseudodifferential operators are bounded on $L^p$, and more generally on $W^{s,p}$, {\em for $1<p<+\infty$}. More precisely, we have the following result.
\begin{prop}[see e.g. \protect{\cite[Section VI.5.2]{Stein}} ]\label{conti}
Let $A = Op(a) \in \Psi^0_{\text{cl}}$. Then, for all $1< p < +\infty$, $s\in \mathbb{R}$, the operator $A$ is continuous on $W^{s,p}( \mathbb{R}^d)$. Furthermore, its norm is bounded by a finite number of semi-norms
$$ \exists N(d); \| Op(a) \|_{\mathcal{L}(W^{s,p} ( \mathbb{R}^d))} \leq C \sup_{|\alpha| + |\beta|+ |\gamma|\leq N(d)}  \|a\|_{0, \alpha, \beta, \gamma}.$$
\end{prop}

 In general, pseudodifferential operators are not continuous on $L^1$ and $L^\infty$, the basic example being the Hilbert transform (smoothed out near $\xi =0$), associated to the symbol
$$a(x,y,\xi) = \chi(\xi), \chi \in C^\infty( \mathbb{R}), \chi \mid_{(-\infty,-1) } =0, \chi \mid_{(1, +\infty) } =1.$$
We also have the following counter example (see the introduction). Let $\chi\in C^\infty_0( \mathbb{R}^2)$ equal to $1$ near $0$, and 
\begin{multline}
u(x):= \log\log( e{|x|^{-1}})\in W^{1,1}(B_1), w: = \chi u \in W^{1,1}_{comp}(\mathbb R^2)\\
\Rightarrow 
(- \Delta +1) w=\frac{\chi(x)} {|x|^2\log^2( e{|x|^{-1}})} + [\chi,  \Delta] u+w = f   \in L^1_{comp}(B_1),\quad D^2 w \notin L^1_{loc}.\end{multline}
We thus get 
\begin{equation}\label{contre-ex}
\frac{ D^2} {(- \Delta +1)^{-1}} f \notin L^1_{loc}, \qquad f \in L^1_{comp}.
\end{equation}   
However,  by the dual estimate of the Lemma in~\cite[Section VI.5.3.1]{Stein}, we have the following weaker result (we shall use only the case $p = 1$).
 \begin{prop}\label{compact}
Let $\delta \in \mathbb{R}, \epsilon >0$ and $A = Op(a) \in \Psi^{\delta}$. Then for all $1\leq p \leq +\infty, s\in \mathbb{R} $ the operator $A$ is continuous from  $W^{s,p}( \mathbb{R}^d)$ to $W^{s- \delta- \epsilon ,p}( \mathbb{R}^d)$. Furthermore, its norm is bounded by a finite number of semi-norms
$$ \exists N(d, \epsilon),\, \| Op(a) \|_{\mathcal{L}(W^{s,p}( \mathbb{R}^d),W^{s- \delta- \epsilon ,p}( \mathbb{R}^d)} \leq C \sup_{|\alpha| + |\beta|+ |\gamma|\leq N(d)}  \|a\|_{\delta, \alpha, \beta, \gamma}.$$
\end{prop}
\begin{corol} \label{coro}
Let $\delta \in \mathbb{R}, \eta >0$ and $A = Op(a) \in \Psi^{\delta}$. Then, for all $s\in \mathbb{R} $, the operator $A$ is continuous from  $\mathcal{M}_0( \mathbb{R}^d)$ to $W^{- \delta- \eta ,1}( \mathbb{R}^d)$. Furthermore, its norm is bounded by a finite number of semi-norms
$$ \exists N(d, \eta),\, \| Op(a) \|_{\mathcal{L}(\mathcal{M}_0( \mathbb{R}^d),W^{- \delta- \eta ,1}( \mathbb{R}^d)} \leq C \sup_{|\alpha| + |\beta|+ |\gamma|\leq N(d)}  \|a\|_{\delta, \alpha, \beta, \gamma}.$$
\end{corol}
Indeed, from the continuous  inclusion for all $\eta >0$,
$$ W^{\eta/2, \infty} ( \mathbb{R}^d) \subset C^0( \mathbb{R}^d),$$
we deduce by duality the continuous inclusion 
$$ \mathcal{M}_0 \subset W^{-\eta/2, 1} ( \mathbb{R}^d), $$
and Corollary~\ref{coro} follows from Proposition~\ref{compact} with $s= - \eta /2$ and $\epsilon = \eta/2$.

We end this section with a result involving weights and bounded measures. 
 \begin{prop}\label{compactbis}
Let $\delta  \in \mathbb{R}$ $\epsilon>0$, $A\in \Psi^{\delta }$ and $\chi  \in C^\infty_0 ( \mathbb{R}^d)$. Then, for any $k \in \mathbb{R}^+$,  the operator 
$$ \chi A   (1+ |x| ^k) $$ is continuous from $ \mathcal{M}_0 ( \mathbb{R}^d) $ to $W^{-\delta- \epsilon, 1} ( \mathbb{R}^d)$.  Furthermore, its norm is bounded by a finite number of semi-norms of $a$:
$$
 \exists N(d, \epsilon),\, \| \chi A  (1+ |x| ^k)  \|_{\mathcal{L}(\mathcal{M}_0, W^{-\delta,1}( \mathbb{R}^d))} 
 \leq C (1+ \sup_{|\alpha|+ |\beta|+ |\gamma| \leq N(d)}  \|a\|_{\delta, \alpha, \beta, \gamma}).
$$
\end{prop}
\begin{proof}
With respect to Corollary~\ref{coro}, the only new point is the presence of the weight and of the cut-off. 
We use a dyadic partition of unity 
$$ 1 =\sum_{p\geq 0} \phi_p (x) , \, \phi_0 \in C^\infty _ 0 (\mathbb{R}^d), \,\forall p \geq 1, \phi_p (x) =  \phi( 2^{-p} x), \,\phi \in C^\infty_0 ( \{ \frac 1 2 < \| x\| <2 \} ),$$
 and write
 $$ \chi  A  (1+ |x| ^k) = \sum_{p\geq 0}  \chi  A  (1+ |x| ^k) \phi_p (x).$$
 According to Corollary~\ref{coro} each term is bounded from $ \mathcal{M}_0 ( \mathbb{R}^d) $ to $W^{-\delta- \epsilon, 1} ( \mathbb{R}^d)$, and we just have to check that the series of the norms is summable. Consider $K_p(x,y) $ the kernel of the operator   $\chi  A  (1+ |x| ^k) \phi_p (x)$:
\begin{equation}\label{kernel}
K_p(x,y)=  \frac 1 {(2\pi)^d} \int e^{i(x-y) \cdot \xi }  a(x,y,\xi)   d\xi \,\chi(x)(1+ |y| ^k)\phi_p(y).
 \end{equation}
Remark that on the support of this kernel, $\| x\| \leq C, \| y \| \sim 2^{p-1},$ and consequently, for $k$ large enough, $\| x-y \| \sim 2^{p}$. Integrating by parts $N$ times  in~\eqref{kernel}
using the identity
$$ L(e^{i(x-y) \cdot \xi} )= - e^{i(x-y) \cdot \xi}, \qquad  L= \frac {i (x-y) \cdot \nabla_\xi} {\| x-y\|^2} ,$$
we get
\begin{equation}\label{kernelbis}
K_p(x,y)=  \frac 1 {(2\pi)^d} \int e^{i(x-y) \cdot \xi } L^N (a(x,y,\xi) ) d\xi\, \chi(x)  (1+ |y| ^k)\phi_p(y),
 \end{equation}
and consequently 
$$ \chi  A  (1+ |x| ^k) \phi_p (x) = \text{Op} (a_{N,p}),\, a_{N,p}=  L^N \bigl(a(x,y,\xi) ) \chi(x) \phi_p(y) (1+ |y| ^k) \in S^{\delta-N} ( \mathbb{R}^d), 
$$ with semi-norms in $S^{\delta-N} $ bounded by 
$$ \| a_{N,p}\|_{\delta-N, \alpha, \beta, \gamma} \leq  C_{N, \alpha, \beta, \gamma}2^{p(k-N)}  \sum_{|\alpha'| \leq |\alpha|, |\beta'| \leq |\beta|, |\gamma'| \leq |\gamma| + N}  \| a\|_{\delta, \alpha', \beta', \gamma'}.
$$
We deduce from Corollary~\ref{coro} that $ \chi  A  (1+ |x| ^k) \phi_p$ is bounded from $\mathcal{M}_0( \mathbb{R}^d)$ to $W^{N-\delta-\eta,1} ( \mathbb{R}^d)$ by $C_{N,\eta} 2^{p(k-N)} $ and we conclude by  choosing $N> k$. 
 
\end{proof}
\section{Temperance}
To deal with pseudodifferential operators in the passage to the limit when defining tangent measures, we need, in the definition of tangent measures,  a temperance property which is actually satisfied on a set of full measure. This property is implicit in the construction by Preiss (see~\cite[Theorem 2.5]{Pr}). A slightly weaker upper-bound \eqref{dyadicbound} can be found in  \cite[Proposition 10.5]{Ribook} with $\eta=2$. As we also need the lower bound we include its short proof here.
 \begin{prop}\label{Preiss}
Let $\nu$ be a compactly supported bounded non negative Radon measure on $\mathbb{R}^d$, and let  $\{r_j\}_{j\in\mathbb N}$ a sequence  convergent to zero.  Then for  $\nu $-a.e. points $x_0$ and every $\eta >0$, there is a subsequence of $\{r_{j_k}\}_{k\in\mathbb N}$  and $C, c>0$ such that
$$ \nu_k :=\frac{(T^{x_0,r_{j_k}})_\sharp\nu}{\nu (B(x_0,r_{j_k}))} , \qquad T^{x_0,r_{j_k}}: y \mapsto x_0+ r_{j_k} x,$$ 
satisfies the uniform bounds  for $k\in\mathbb N$ and $R>1$:
\begin{equation}\label{dyadicbound}  \nu_k(B(x_0,R))=\frac{\nu (B(x_0,R r_{j_k}))}{\nu (B(x_0,r_{j_k}))}\leq C R^{d+\eta}, \end{equation}
\begin{equation}\label{dyadicboundbis}   \nu_k(B(x_0,R^{-1} ))=\frac{\nu (B(x_0, R^{-1} r_{j_k} ))}{\nu (B(x_0,r_{j_k}))}\geq c R^{-(d+\eta)}. \end{equation}
\end{prop}
\begin{rem}
The proof gives actually a more precise bound (with a logarithmic loss)  $R^d \log ^{a} (1+R)$, $a>1$. On the other hand, for absolutely continuous measures $\nu = fd\mathcal L^d$, at all Lebesgue points of $f$, i.e. for $\mathcal L^d$-almost every $x_0$ we have
$$ \lim_{r\rightarrow 0} \frac{ 1} {C_d r^d} \int_{B(x_0,r)} f(x) dx = f(x_0). $$ 
Since $0< |f(x_0)| <+\infty$ for $fd\mathcal L^d$-almost every $x_0$ we get easily that in~\eqref{dyadicbound}~-\eqref{dyadicboundbis} we can replace $R^{\pm (d+ \eta)}$ by $R^{\pm d}$ (without the logarithmic loss). We do not know if it is the case for general~$\nu$, or even for doubling measures $\nu$.
\end{rem}

\begin{proof}
We basically follow the construction of tangent measures in ~\cite{Pr} with minor differences, but since the temperance bound~\eqref{dyadicbound} does not seem to appear anywhere explicitly in the literature, we give below the complete proof.

We can use Fubini and get for any set $A$ and radius $R$
\begin{equation}\label{fubini}
\nu(A)=\frac1{\mathcal L^d(B(0,R))}\int \nu(A\cap B(x,R))dx.
\end{equation}
We consider for $r>0, \delta>0, k\in\mathbb N$ the set,
$$E_{r,k,\delta}:=\{x\in\mathbb R^d, \nu(B(x,2^   kr))\geq \beta_{k,\delta} \, \nu (B(x,r))\},$$
where
$$\beta_{k,\delta}:=\frac{2^d(2^k+1)^d)\,\nu(K)}{\delta}, $$
$K$ is a compact set where $\nu$ is supported, and $\delta>0$ to be chosen later.

Applying~\eqref{fubini} to $A= E_{r,k,\delta}$, $R= \frac r 2$, we get 
$$\nu (E_{r,k,\delta})=\frac1{\omega_d(\frac r2)^d}\int \nu(E_{r,k,\delta}\cap B(x,\frac r2))dx,$$
where $\omega_d=\mathcal L^d(B(0,1))$.
If $E_{r,k,\delta}\cap B(x,\frac r2)\neq\varnothing$ then we can use a point $z$ in this intersection to get
\begin{multline}
\nu(E_{r,k,\delta}\cap B(x,\frac r2))\leq \nu(B(x,\frac r2))\leq \nu(B(z,r))\\
\leq \frac1{\beta_{r,k,\delta}} \nu(B(z,2^k r))\leq \frac 1{\beta_{r,k,\delta}}  \nu(B(x,(2^k+1)r)).\end{multline}
Therefore for any $r>0, \delta>0, k\in\mathbb N$, and using again \eqref{fubini} with $A=K$ we obtain
\begin{equation}\label{estsize}
\nu(E_{r,k,\delta})\leq \delta \frac 1{\omega_d((2^k+1)r)^d\nu(K)}\int \nu(B(x,(2^k+1)r))dx=\delta,
\end{equation}
and on the complementary set $^cE_{r,k,\delta}$ we have
\begin{equation}\label{estsup}
\frac{\nu(B(x,2^kr))}{ \nu(B(x,r))}\leq \beta_{k,\delta}=\frac{2^d(2^k+1)^d\,\nu(K)}{\delta},
\end{equation}
while on the complementary set $^cE_{r 2^{-k},k, \delta}$ we have similarly 
\begin{equation}\label{estinf}\frac{\nu(B(x,r))}{ \nu(B(x,2^{-k}r))}\leq \beta_{k,\delta}=\frac{2^d(2^k+1)^d\,\nu(K)}{\delta}. 
\end{equation}

Let $r_j$ be a sequence convergent to zero and $\epsilon >0$. 
Now we consider 
$$A_{\{r_j\},\epsilon}:=\cup_{i=1}^{\infty}\cap _{j=i}^\infty A_{r_j,\epsilon},$$ where $A_{r_j,\epsilon}:=\cup_{k=1}^\infty \bigl( E_{r_j,k, a_k\epsilon}\cup  E_{r_j2^{-k}, k,a_k \epsilon}\bigr)$ and $a_k= k^{-1-\gamma}, \gamma>0$, summable  so that in view of \eqref{estsize} we have $\nu(A_{r_j,\epsilon})\leq C\epsilon$. In particular $\nu(A_{\{r_j\},\epsilon})\leq C \epsilon$. For any point $x\notin A_{\{r_j\},\epsilon}$ we get for all $i\in\mathbb N$ the existence of $j\geq i$ such that $x\notin A_{r_j,\epsilon}$. This yields a subsequence $r_{j_n}$ such that $x\notin A_{r_{j_n},\epsilon}$ for all $n\in\mathbb N$, thus $x\notin E_{r_{j_n},k, a_k\epsilon}$ and $x\notin E_{r_{j_n}2^{-k},k, a_k\epsilon}$  for all $n,k\in\mathbb N$. Then, renaming this sequence $r_j$ for simplicity, from \eqref{estsup}-\eqref{estinf} we get 
$$\frac{\nu (B(x,2^k r_j))}{\nu(B(x,r_j))}\leq \frac{2^{kd}}{k^{1+\gamma}\epsilon},\quad \frac{\nu(B(x,2^{-k} r_j))}{\nu (B(x, r_j))}\geq \frac{k^{1+\gamma}\epsilon}{2^{kd}},\quad  \forall j,k\in \mathbb N.$$
As a consequence for all $x\notin A_{\{r_j\},\epsilon}$ there is a subsequence along which a non-zero tangent measure at $x$ is obtained. By defining 
$$A_{\{r_j\}}:=\cap_{n=1}^\infty A_{\{r_j\},\frac 1n},$$
we get a set of zero measure such that for all $x\notin A_{\{r_j\}}$, thus in particular $x\notin A_{\{r_j\},\frac 1{n_x}}$ for some $n_x\in\mathbb N$, there is a subsequence along which a non-zero tangent measure at $x$ is obtained and the bounds \eqref{dyadicbound}-\eqref{dyadicboundbis} are satisfied for $R= 2^k$, and hence for all $R\geq 1$.
\end{proof}

 \end{document}